\PassOptionsToPackage{dvipsnames}{xcolor}
\documentclass{article}

\usepackage{amsmath,amsthm,amsfonts,amssymb}
\usepackage{todonotes}
\usepackage{mathtools}
\usepackage{latexsym}
\usepackage{tikz,tikz-cd}
\usepackage{comment}
\usepackage{tikzit}				

\tikzstyle{morphism}=[fill=white, draw=black, shape=rectangle]
\tikzstyle{medium box}=[fill=white, draw=black, shape=rectangle, minimum width=0.7cm, minimum height=0.7cm]
\tikzstyle{large morphism}=[fill=white, draw=black, shape=rectangle, minimum width=1.7cm, minimum height=1cm]
\tikzstyle{bn}=[fill=black, draw=black, shape=circle, inner sep=1.5pt]
\tikzstyle{state}=[fill=white, draw=black, regular polygon, regular polygon sides=3, minimum width=0.8cm, shape border rotate=180, inner sep=0pt]
\tikzstyle{medium state}=[fill=white, draw=black, regular polygon, regular polygon sides=3, minimum width=1.3cm, inner sep=0pt, shape border rotate=180]
\tikzstyle{large state}=[fill=white, draw=black, regular polygon, regular polygon sides=3, minimum width=2.2cm, shape border rotate=180, inner sep=0pt]
\tikzstyle{wide state}=[fill=white, draw=black, shape=isosceles triangle, minimum width=0.8cm, shape border rotate=270, inner sep=1.4pt, minimum height=0.5cm, isosceles triangle apex angle=80]
\tikzstyle{wn}=[fill=white, draw=black, shape=circle, inner sep=1.5pt]
\tikzstyle{blue morphism}=[fill=white, draw={rgb,255: red,15; green,0; blue,150}, shape=rectangle, text={rgb,255: red,15; green,0; blue,150}, tikzit category=blue]
\tikzstyle{blue state}=[fill=white, draw={rgb,255: red,15; green,0; blue,150}, shape=circle, regular polygon, regular polygon sides=3, minimum width=0.8cm, shape border rotate=180, inner sep=0pt, text={rgb,255: red,15; green,0; blue,150}, tikzit category=blue]
\tikzstyle{blue node}=[fill={rgb,255: red,15; green,0; blue,150}, draw={rgb,255: red,15; green,0; blue,150}, shape=circle, tikzit category=blue, inner sep=1.5pt]
\tikzstyle{blue}=[text={rgb,255: red,15; green,0; blue,150}, tikzit draw={rgb,255: red,191; green,191; blue,191}, tikzit category=blue, tikzit fill=white, inner sep=0mm]
\tikzstyle{blue wide state}=[fill=white, draw={rgb,255: red,15; green,0; blue,150}, text={rgb,255: red,15; green,0; blue,150}, shape=isosceles triangle, minimum width=0.8cm, shape border rotate=270, inner sep=1.4pt, minimum height=0.5cm, isosceles triangle apex angle=80]
\tikzstyle{red node}=[fill={rgb,255: red,150; green,0; blue,2}, draw={rgb,255: red,150; green,0; blue,2}, shape=circle, inner sep=1.5pt]
\tikzstyle{Purple node}=[fill={rgb,255: red,150; green,0; blue,150}, draw={rgb,255: red,150; green,0; blue,150}, shape=circle, inner sep=1.5pt]
\tikzstyle{red}=[text={rgb,255: red,150; green,0; blue,2}, inner sep=0mm, tikzit fill=white, tikzit draw={rgb,255: red,191; green,191; blue,191}]
\tikzstyle{purple}=[text={rgb,255: red,150; green,0; blue,150}, inner sep=0mm, tikzit fill=white, tikzit draw={rgb,255: red,191; green,191; blue,191}]
\tikzstyle{white morphism}=[fill=white, draw=white, shape=rectangle, tikzit draw={rgb,255: red,139; green,139; blue,139}]

\tikzstyle{arrow}=[->]
\tikzstyle{dashed box}=[-, dashed]
\tikzstyle{blue arrow}=[-, draw={rgb,255: red,15; green,0; blue,150}, tikzit category=blue]
\tikzstyle{mapsto}=[{|->}]
\tikzstyle{double wire}=[-, double]
\tikzstyle{curly brace}=[-, draw=none, tikzit draw={rgb,255: red,128; green,0; blue,128}]

\usetikzlibrary{decorations.pathmorphing,calc}
\usepackage{enumitem}
\usepackage[T1,T2A]{fontenc}			
\usepackage{cite}
\usepackage{tikz-cd}
\usepackage{mathpartir}
\usepackage{stmaryrd}

\definecolor{myurlcolor}{rgb}{0,0,0.3}
\definecolor{mycitecolor}{rgb}{0,0.3,0}
\definecolor{myrefcolor}{rgb}{0.3,0,0}
\usepackage[pagebackref,draft=false]{hyperref}
\hypersetup{colorlinks,
linkcolor=myrefcolor,
citecolor=mycitecolor,
urlcolor=myurlcolor}

\usepackage[capitalize]{cleveref}

\newtheorem{theorem}{Theorem}[section]
\newtheorem{proposition}[theorem]{Proposition}
\newtheorem{lemma}[theorem]{Lemma}
\newtheorem{corollary}[theorem]{Corollary}
\newtheorem{definition}[theorem]{Definition}

\theoremstyle{definition}
\newtheorem{example}[theorem]{Example}

\newcommand{\N}{\mathbb{N}}

\newcommand{\R}{\mathbb{R}}

\newcommand{\inv}{\mathrm{inv}}		

\newcommand{\cat}[1]{{\mathsf{#1}}} 

\newcommand{\id}{\mathrm{id}} 		

\tikzset{pullback/.style={minimum size=1.2ex,path picture={	
			\draw[opacity=1,black,-,#1] (-0.5ex,-0.5ex) -- (0.5ex,-0.5ex) -- (0.5ex,0.5ex);%
}}}

\renewcommand{\det}{\mathrm{det}}	

\newcommand{\Stoch}{\mathsf{Stoch}}

\newcommand{\as}[1]{
		\def\relstate{#1}%
		\ifx\relstate\empty
		  \text{a.s.}%
		\else
		  {#1\text{-a.s.}}%
		\fi
	}
\newcommand{\ase}[1]{=_{#1\text{-a.s.}}}					

\DeclareMathOperator{\cop}{copy}
\DeclareMathOperator{\del}{del}

	\DeclarePairedDelimiterX{\Set}[1]{\{}{\}}{%
		
		#1
	}
	\makeatletter
		\let\oldSet\Set
		\def\Set{\@ifstar{\oldSet}{\oldSet*}}
	\makeatother

	\DeclarePairedDelimiterX{\Family}[1]{(}{)}{%
		
		#1
	}
	\makeatletter
		\let\oldFamily\Family
		\def\Family{\@ifstar{\oldFamily}{\oldFamily*}}

\tikzset{commutative diagrams/kl/.style={rightsquigarrow}}
\newcounter{sarrow}

\title{A category-theoretic proof of the ergodic decomposition theorem}

\author{Sean Moss and Paolo Perrone}

\begin{document}

\maketitle

\begin{abstract}
 The ergodic decomposition theorem is a cornerstone result of dynamical systems and ergodic theory. It states that every invariant measure on a dynamical system is a mixture of ergodic ones.
 Here we formulate and prove the theorem in terms of string diagrams, using the formalism of Markov categories. 
 We recover the usual measure-theoretic statement by instantiating our result in the category of stochastic kernels.
 Along the way we give a conceptual treatment of several concepts in the theory of deterministic and stochastic dynamical systems. 
 In particular,
 \begin{itemize}
  \item ergodic measures appear very naturally as particular cones of deterministic morphisms (in the sense of Markov categories);
  \item the invariant $\sigma$-algebra of a dynamical system can be seen as a colimit in the category of Markov kernels.
 \end{itemize}
 In line with other uses of category theory, once the necessary structures are in place, our proof of the main theorem is much more intuitive than traditional approaches. In particular, it does not use any quantitative limiting arguments, and it does not rely on the cardinality of the group or monoid indexing the dynamics.
 We hope that this result paves the way for further applications of category theory to dynamical systems, ergodic theory, and information theory.
\end{abstract}

\tableofcontents

\section{Introduction}

In recent years there has been growing interest in the mathematics and computer science communities about expanding the formalism of measure theory. The goal is to capture a wider range of structures and phenomena in probability theory and related fields such as statistics and information theory. 

On one hand, there is interest in moving measure theory beyond the need of countability, for example in the work of Jamneshan and Tao~\cite{uncountable-measure-theory}, and in their work with others in expanding ergodic theory in that direction~\cite{uncountable-19,uncountable-20,uncountable-21,uncountable-22}.

In addition, there has been work in translating the basic ideas of probability theory into an abstract, axiomatic formalism, of which the traditional measure-theoretic probability is a concrete instance. This approach is sometimes called \emph{categorical probability}, and it is mostly done by means of \emph{Markov categories}. 
In their current form they were defined in~\cite{fritz2019synthetic}, with some of the concepts already present in earlier work such as~\cite{chojacobs2019strings} for ``GS monoidal'' or ``CD'' categories, a slightly more general structure. (The first definitions, in a different context, date back at least to~\cite{gadducci-thesis} --- see also~\cite[Section~1]{freeGScat} for a more detailed overview.)

Categorical probability is an example of a \emph{synthetic} theory, as opposed to analytic. To clarify the terms, here is a classical analogy. The geometry of the plane can be studied \emph{synthetically} starting from axioms such as Euclid's ones, or one can do \emph{analytic geometry}, in the sense of Descartes, doing calculations in coordinates. Since $\R^2$ satisfies Euclid's axioms, analytic geometry is indeed a model of Euclidean geometry, and in order to prove a theorem, one could use either approach, both methods having advantages and disadvantages.

Similarly, with Markov categories one first formulates some fundamental axioms for probability theory. The theorems of traditional probability theory and statistics can then be recast in a more general and abstract categorical form, and proven purely in terms of these axioms, focusing on the conceptual aspects, and without relying on the specific properties of the objects of the category in question (such as cardinality or separability). 
The category $\cat{Stoch}$ of Markov kernels (see \Cref{basicmarkov}) is an example of a Markov category, and one can obtain the traditional results of probability by instantiating the abstract versions in $\cat{Stoch}$ or in one of its subcategories.
Several results of probability theory have recently been reproven in this way, for example the Kolmogorov and Hewitt-Savage zero-one laws~\cite{fritzrischel2019zeroone} and the de Finetti theorem~\cite{definetti-markov,ours_LICS}. 

In computer science there is interest in finding an alternative to, or an extension of, traditional measure theory, in order to talk about \emph{random functions}, in the sense of \emph{random elements of a function space}. This is known to be impossible in traditional measure theory~\cite{notccc}, and so additional theory is needed. A recently defined structure which solves that problem is \emph{quasi-Borel spaces}~\cite{heunen-kammar-staton-yang-a-convenient-category-for-higher-order-probability-theory}. Quasi-Borel spaces have interesting properties which may sound counterintuitive if one comes from traditional measure theory, and the Markov category formalism helps to elucidate the conceptual differences~\cite{sabok-staton-stein-wolman-probabilistic-programming-semantics-for-name-generation}.

In this work we start expanding the approach of categorical probability to ergodic theory. We focus on one particular result, the \emph{ergodic decomposition theorem}, which can be roughly stated as ``every invariant measure of a deterministic dynamical system can be written as a convex mixture of ergodic ones''~\cite[Section~5]{viana-oliveira}.
While the traditional form of the theorem, relying on the notions of convexity and almost-sure equality, seems to be very specific to the measure-theoretic formalism, we show that both the statement and the proof can be rewritten in terms of category theory, with most of the conceptual steps used in the proof (such as disintegrations of measures) being already well studied in terms of Markov categories. 
In particular, the notion of ``convex mixture'' can be interpreted as categorical composition (see~\Cref{mixtures}).

An interesting feature of this approach is that the very definition of ergodic measure sits very naturally within the Markov category formalism. 
Indeed, Markov categories come with a notion of \emph{deterministic states} (see \Cref{defdet}) which, when instantiated in $\cat{Stoch}$, give exactly the \emph{zero-one measures}, those probability measures which assign to each event probability either zero or one.
Since ergodic measures are traditionally (equivalently) defined as measures which are zero-one on the $\sigma$-algebra of invariant sets, we can redefine them categorically, and more generally, as particular deterministic states (see~\Cref{ergodicstates}). Invariant sets also have a natural categorical characterization, since the invariant $\sigma$-algebra satisfies a particular universal property in the category of Markov kernels (see \Cref{markovquotient} and \Cref{weakquotient}).

In the case of deterministic dynamical systems one can talk about ergodic measures either in terms of sets which are invariant in the strict sense (see~\Cref{definvset}), or only up to measure zero \cite[Theorem~5.1.3]{viana-oliveira}, and the notions are equivalent for a large class of systems (see for example \cite[Theorem~3]{tao}). 
In this work we focus on the strict approach. 
A Markov-categorical formalism to treat morphisms up to almost-sure equality exists~\cite[Definition~13.8]{fritz2019synthetic}, and it may open an interesting new approach, which for now we leave to future work.

\paragraph{Outline.}
In \Cref{background} we start by giving some general background by explaining how to write some notions of dynamical systems in a category-theoretic way (\Cref{dynsyscat}). We then recall the main definitions and constructions of Markov categories (\Cref{basicmarkov}), and use them to express some concepts of probability (\Cref{probmarkov}) and dynamical systems (\Cref{dynsysmarkov}). Most of the material here is well known, except from the last section.

The main original contributions of this work are in \Cref{ergodic}. In particular, in \Cref{mixtures} we give a categorical definition of ``mixtures'' or ``convex combinations''. In \Cref{markovquotient} we give a categorical characterization of invariant sets, giving a definition that can work in more general categories than Markov kernels.
In \Cref{ergodicstates} we then express ergodicity of states categorically, in a way that generalizes the usual definition of ``assuming only values zero and one on invariant sets''.
Our main result, a synthetic version of the ergodic decomposition theorem (\Cref{mainthm}), is stated and proven in \Cref{mainsec}, together with its instantiation in traditional measure theory (\Cref{mainstoch}). 

In \Cref{weakquotient} make mathematically precise the intuition that ``the invariant $\sigma$-algebra is a weak analogue of a space of orbits'', using universal properties in the category of Markov kernels.

\paragraph{Acknowledgements.}
We would like to thank Tobias Fritz, Tomáš Gonda and Dario Stein for the interesting discussions on Markov categories, and Sharwin Rezagholi for the inspiring conversations on dynamical systems and ergodic theory.
We would also like to thank Sam Staton and his research group for the support and for the helpful feedback.

\section{Background}\label{background}

Some aspects of the theory of dynamical systems lend themselves very well to a category-theoretic treatment. 
Here we first look at some of the ideas of dynamical systems that can be formalized in terms of categories, in particular the ideas of invariant states and observables, that can be thought of as particular cones and cocones. (There is more to be said about dynamical systems and categories, a good starting point could be \cite{dresden}.)

We then recall the basic definitions and results of Markov categories, which will be used in the rest of this work, and we review some of the main probabilistic concepts which can be expressed in terms of Markov categories. In \Cref{dynsysmarkov} we turn to some structures involving dynamical systems in Markov categories which are, as far as we are aware, first defined in this work. 

\subsection{Dynamical systems and categories}\label{dynsyscat}

In ergodic theory and related fields, one is mostly interested in the following two types of dynamical system:
\begin{enumerate}
 \item A set or space $X$ with some structure (topology, measure, etc.), and a map or kernel $t:X\to X$ preserving that structure;
 \item A set or space $X$ with some structure, and a group acting on $X$ in a structure-preserving way.
\end{enumerate}
Both dynamics are encompassed by the notion of a monoid: every group is a monoid, and every map $t:X\to X$ generates a monoid via iterations, $\{\id,t,t^2,\dots\}$.

In terms of category theory, a monoid is equivalently a category with a single object. Given a monoid $M$, denote by $\cat{B}M$ the category with a single object, denoted by $\bullet$, and with the set of arrows $\bullet\to\bullet$ given by $M$, with its identity and composition. 

Let now $\cat{C}$ be a category (for example, the category $\cat{Meas}$ of measurable spaces and measurable maps). 
A dynamical system can be modeled as a functor $\cat{B}M\to \cat{C}$. Let's see this explicitly. Such a functor maps
\begin{itemize}
 \item The unique object $\bullet$ of $\cat{B}M$ to an object $X$ of $\cat{C}$ (for example, a measurable space). This is the object (or ``space'') where the dynamics takes place;
 \item Each arrow of $\cat{B}M$, i.e.~each element $m$ of the monoid $M$, to an ``induced'' arrow $X\to X$ (for example, a measurable map). This is the dynamics. 
\end{itemize}
For brevity, we denote the dynamical system just by $X$ whenever this does not cause ambiguity, and we denote the map $X\to X$ induced by $m\in M$ again by $m$, writing $m:X\to X$. 

If $M=\N$, the monoid is generated by the number $1$, and so the dynamics is generated by the arrow induced by $1$ (for example, a measurable map), which then is iterated. We usually denote the resulting map by $t:X\to X$. 

Here are other examples of categories $\cat{C}$:
\begin{itemize}
 \item If $\cat{C}$ is the category of compact Hausdorff spaces and continuous maps, a functor $\cat{B}M\to \cat{C}$ is a (compact) topological dynamical system.
 \item If $\cat{C}$ is the category of measure spaces and measure-preserving maps, a functor $\cat{B}M\to \cat{C}$ is a measure-preserving dynamical system.
 \item If $M=\N$ and $\cat{C}$ is the category of measurable spaces and Markov kernels, a functor $\cat{B}M\to \cat{C}$ is a (discrete-time) Markov chain.
\end{itemize}

One of the most useful contributions of the categorical formalism to dynamical systems is a systematic treatment of invariant states and observables. 
Consider a dynamical system on the object $X$ indexed by the monoid $M$. A \emph{cone} over $X$ is an object $C$ together with an arrow $c:C\to X$ such that for every $m\in M$, the following diagram commutes.
\begin{equation}\label{cone}
\begin{tikzcd}[column sep=small]
 & C \ar{dl}[swap]{c} \ar{dr}{c} \\
 X \ar{rr}[swap]{m} && X
\end{tikzcd}
\end{equation}
In general, cones over a given dynamical system have the interpretation of ``invariant states of some kind''. Indeed, in the category of sets and functions, if $C$ is a one-point set, the arrow $c:C\to X$ is the inclusion of some point $x\in X$, and the diagram \eqref{cone} just says that $x$ is a fixed point, i.e.\ $m(x)=x$. 
More generally, if $C$ is not a one-point set, the arrow $c:C\to X$ selects a $C$-indexed family of fixed points in $X$.
(Note that what matters here is that each individual point $c(i)$ is a fixed point, rather than any property of the range $\{ c(i) : i \in C \} \subseteq X$ of $c$ as a subset.)

Since the diagram \eqref{cone} can be written as $m\circ c = c$, sometimes one says that $m$ is a \emph{left-invariant} morphism (since $m$ acts on the left of $c$).
Note that the commutativity of diagram \eqref{cone} needs only to be checked on generators of $M$. For example, for $M=\N$, it suffices to check the condition for $1\in\N$, the map that we usually denote by $t:X\to X$.

A cone $C\to X$ over $X$ is \emph{universal}, or a \emph{limit}, if for every (other) cone $D\to X$ there is a unique arrow $D\to C$ such that for every $m\in M$, the following diagram commutes. 
$$
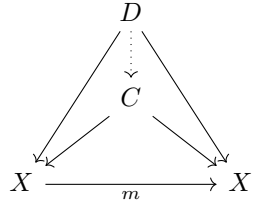
\begin{tikzcd}
 & D \ar{ddl} \ar{ddr} \ar[dotted]{d} \\
 & C \ar{dl} \ar{dr} \\
 X \ar{rr}[swap]{m} && X
\end{tikzcd}
$$
The limit cone $C$, if it exists, is unique up to isomorphism, and can be interpreted as the ``largest subspace of $X$ of invariant states''. In the category of sets, it is precisely the set of all invariant points (which can be empty), and in other categories it has a similar interpretation.
We denote the limit, if it exists, by $X^\inv$. 

Dually, a \emph{cocone} under the dynamical system $X$, or a \emph{right-invariant} morphism, is an object $R$ together with an arrow $r:X\to R$ such that for all $m\in M$, $f\circ m= f$, i.e.~the following diagram commutes. 
$$
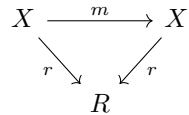
\begin{tikzcd}[column sep=small]
 X \ar{rr}{m} \ar{dr}[swap]{r} && X \ar{dl}{r} \\
 & R
\end{tikzcd}
$$
This has the interpretation of an \emph{invariant function} or \emph{invariant observable}. In the category of sets, this is precisely a function with the property that $r(mx)=r(x)$, i.e.~it is constant on each orbit.

A cocone is universal, or a \emph{colimit}, if for every (other) cocone $X\to S$ there is a unique arrow $R\to S$ such that for every $m\in M$, the following diagram commutes.
$$
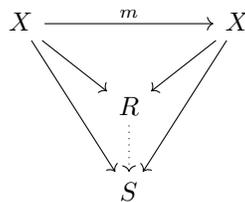
\begin{tikzcd}
 X \ar{rr}{m} \ar{dr} \ar{ddr} && X \ar{dl} \ar{ddl} \\
 & R \ar[dotted]{d} \\
 & S
\end{tikzcd}
$$
Again, this object, if it exists, is unique up to isomorphism, and it can be interpreted as the ``finest invariant observable''. In the category of sets, it is precisely the set of orbits, and the commutation of the last diagram says precisely that every invariant observable factors through the orbit.
In other categories the interpretation is similar (for example, for compact Hausdorff spaces one obtains the quotient by the smallest \emph{closed} equivalence relation which contains the orbits). 
We denote the colimit $X$, if it exists and up to isomorphism, by $X_\inv$.

Further categorical formalism for dynamical systems, similar in spirit to this section, can be found in \cite{dresden}, for the case of cartesian closed categories.

\subsection{Basic concepts of Markov categories}\label{basicmarkov}

Markov categories are a category-theoretic framework for probability and related fields. They allow us to express several conceptual aspects of probability theory (such as stochastic dependence and independence, almost-sure equality, and conditional distributions) in a graphical language, where the formalism takes care automatically of the measure-theoretic aspects. See~\cite{fritz2019synthetic} for more details.

The basic idea of a Markov category, which we will define shortly, is that of a category whose morphisms are ``probabilistic maps'' or ``transitions''. 
One of the most basic examples is the category $\cat{FinStoch}$, where
\begin{itemize}
 \item Objects are finite sets, which we denote by $X$, $Y$, etc.;
 \item Morphisms are \emph{stochastic matrices}. A stochastic matrix from $X$ to $Y$ is a function 
 $$
 \begin{tikzcd}[row sep=0]
  X\times Y \ar{r}{k} & {[0,1]} \\
  (x,y) \ar[mapsto]{r} & k(y|x)
 \end{tikzcd}
 $$
 such that for all $x\in X$ we have $\sum_{y\in Y} k(y|x)=1$.
 A possible interpretation is a transition probability from state $x$ to state $y$. 
 \item The composition of stochastic matrices is equivalently the \emph{Chapman-Kolmogorov formula}. For $k:X\to Y$ and $h:Y\to Z$, 
 $$
 h\circ k(z|x) = \sum_{y\in Y} h(z|y) \, k(y|x) .
 $$
\end{itemize}

The most important example is the category $\cat{Stoch}$, where
\begin{itemize}
 \item Objects are measurable spaces, which we denote as either $(X,\Sigma_X)$ or more briefly as $X$;
 \item Morphisms are \emph{Markov kernels}. A Markov kernel from $X$ to $Y$ is a function 
 $$
 \begin{tikzcd}[row sep=0]
  X\times \Sigma_Y \ar{r}{k} & {[0,1]} \\
  (x,B) \ar[mapsto]{r} & k(B|x)
 \end{tikzcd}
 $$
 such that 
 \begin{itemize}
  \item For each $x\in X$, the assignment $B\mapsto k(B|x)$ is a probability measure on $Y$;
  \item For each $B\in\Sigma_Y$, the assignment $x\mapsto k(B|x)$ is a measurable function on $X$.
 \end{itemize}
 A possible interpretation of the quantity $k(B|x)$ is the ``probability that the next state is in $B$ if the current state is $x$''.
 \item The composition of Markov kernels is given by the integral version of the Chapman-Kolmogorov formula. For $k:X\to Y$ and $h:Y\to Z$, and for each measurable $C\in \Sigma_Z$,
 $$
 h\circ k(C|x) = \int_{Y} h(C|y) \, k(dy|x) .
 $$
\end{itemize}
Every measurable function $f:X\to Y$ defines a ``deterministic'' Markov kernel $K_f$ as follows.
\begin{equation}\label{kernelfrommeasure}
F_f(B|x) \coloneqq \begin{cases}
                    1 & f(x) \in B \\
                    0 & f(x) \notin B
                   \end{cases}
\end{equation}
for each $x\in X$ and $B\in\Sigma_Y$.
This construction defines then a functor $K:\cat{Meas}\to\cat{Stoch}$ from measurable functions to Markov kernels.

Markov categories can be considered an abstraction of the category $\cat{Stoch}$, where the main categorical properties are formulated as axioms, and used to prove theorems of probability without having to use measure theory directly. 

One of the main structures of the categories $\cat{Stoch}$ and $\cat{FinStoch}$, which figures prominently in the definition of a Markov categories, is the concept of a \emph{monoidal} or \emph{tensor product}. 
The basic idea is that sometimes one wants to consider two systems, and talk about joint or composite states. Sometimes, the systems transition independently, and sometimes they interact. The categorical notion of a tensor product (generalizing, for example, the usual tensor of vector spaces) models this idea.
Indeed, given objects $X$ and $Y$, we want to form a ``composite'' object, denoted by $X\otimes Y$. In $\cat{Stoch}$, we take the cartesian product of measurable spaces, with the product $\sigma$-algebra. Moreover, given morphisms $k:X\to Y$ and $h:Z\to W$, we want a morphism $k\otimes h:X\otimes Z\to Y\otimes W$, with the interpretation that in this case, the dynamics is given by $k$ and $h$ independently on the two subsystems. This means that the tensor product is a functor of two variables, $\otimes:\cat{C}\otimes\cat{C}\to \cat{C}$.
In \emph{Stoch}, this is given by the (independent) product of Markov kernels, 
$$
k\otimes h(B,D|x,z) = k(B|x) \, h(D|z) 
$$
for all $x\in X$, $z\in Z$, $B\in\Sigma_Y$, and $D\in\Sigma_W$. 
It is helpful to use \emph{string diagrams} to represent these products. We draw a morphism $k:X\to Y$ as the following diagram, which should be read from bottom to top.
$$
\begin{tikzpicture}
	\begin{pgfonlayer}{nodelayer}
		\node [style=morphism] (15) at (0, 0) {$k$};
		\node [style=none] (17) at (0, 1) {};
		\node [style=none] (18) at (0, -1.5) {$X$};
		\node [style=none] (19) at (0, 1.5) {$Y$};
		\node [style=none] (30) at (0, -1) {};
	\end{pgfonlayer}
	\begin{pgfonlayer}{edgelayer}
		\draw (17.center) to (15);
		\draw (15) to (30.center);
	\end{pgfonlayer}
\end{tikzpicture}
$$
We can write the tensor product $k\otimes h$ by simply juxtaposing the two morphisms, as follows. 
$$
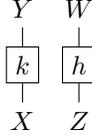
\begin{tikzpicture}
	\begin{pgfonlayer}{nodelayer}
		\node [style=morphism] (15) at (-0.75, 0) {$k$};
		\node [style=none] (17) at (-0.75, 1) {};
		\node [style=none] (18) at (-0.75, -1.5) {$X$};
		\node [style=none] (19) at (-0.75, 1.5) {$Y$};
		\node [style=none] (30) at (-0.75, -1) {};
		\node [style=morphism] (31) at (0.75, 0) {$h$};
		\node [style=none] (32) at (0.75, 1) {};
		\node [style=none] (33) at (0.75, -1.5) {$Z$};
		\node [style=none] (34) at (0.75, 1.5) {$W$};
		\node [style=none] (35) at (0.75, -1) {};
	\end{pgfonlayer}
	\begin{pgfonlayer}{edgelayer}
		\draw (17.center) to (15);
		\draw (15) to (30.center);
		\draw (32.center) to (31);
		\draw (31) to (35.center);
	\end{pgfonlayer}
\end{tikzpicture}
$$
This notation reflects the fact that the two subsystems do not interact. A general morphism between $X\otimes Z$ and $Y\otimes W$ will exhibit interaction and will not be in the form above --- we represent it as follows.
$$
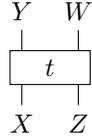
\begin{tikzpicture}
	\begin{pgfonlayer}{nodelayer}
		\node [style=none] (17) at (-0.75, 1) {};
		\node [style=none] (18) at (-0.75, -1.5) {$X$};
		\node [style=none] (19) at (-0.75, 1.5) {$Y$};
		\node [style=none] (30) at (-0.75, -1) {};
		\node [style=none] (32) at (0.75, 1) {};
		\node [style=none] (33) at (0.75, -1.5) {$Z$};
		\node [style=none] (34) at (0.75, 1.5) {$W$};
		\node [style=none] (35) at (0.75, -1) {};
		\node [style=morphism] (36) at (0, 0) {$\quad t\quad$};
	\end{pgfonlayer}
	\begin{pgfonlayer}{edgelayer}
		\draw (17.center) to (30.center);
		\draw (32.center) to (35.center);
	\end{pgfonlayer}
\end{tikzpicture}
$$
Moreover, we need a \emph{unit} object $I$, which accounts for a ``trivial'' state. In $\cat{Stoch}$, this is the one-point measurable space.  
Markov kernels of the form $p:I\to X$ are equivalently just probability measures in $X$. In general, we call morphisms in this form \emph{states}, and denote them as follows.
$$
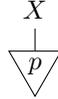
\begin{tikzpicture}
	\begin{pgfonlayer}{nodelayer}
		\node [style=none] (2) at (0, 1) {};
		\node [style=none] (4) at (0, 1.5) {$X$};
		\node [style=state] (5) at (0, 0) {$p$};
	\end{pgfonlayer}
	\begin{pgfonlayer}{edgelayer}
		\draw (5) to (2.center);
	\end{pgfonlayer}
\end{tikzpicture}

$$
We have associativity and unitality isomorphisms which resemble the axioms for a monoid, 
$$
(X\otimes Y)\otimes Z \cong X\otimes(Y\otimes Z) ,\qquad X\otimes I \cong X \cong I\otimes X ,
$$
and so, a category with this notion of product is called a \emph{monoidal category}.
A monoidal category is \emph{symmetric} if each product $X\otimes Y$ is isomorphic to $Y\otimes X$ in a very strong sense, so that for all practical purposes, the order of the factors does not matter. 
For the rigorous definition of a monoidal category, see for example \cite[Section~VII.1]{maclane-cwm1998}.
The categorical product with its usual universal property satisfies the axioms of a monoidal product, in that case one talks about a \emph{cartesian} monoidal category. The categories $\cat{Stoch}$ and $\cat{FinStoch}$ are symmetric monoidal, but not cartesian. In general, whenever randomness is involved, we do not want a cartesian category (see below for more on this).

Here is the rigorous definition of a Markov category. 

\begin{definition}
 A Markov category is a symmetric monoidal category $(\cat{C},\otimes,I)$, or more briefly $\cat{C}$, where
 \begin{itemize}
  \item Each object $X$ is equipped with maps $\cop:X\to X\otimes X$ and $\del:X\to I$, which we call ``copy and delete'' or ``copy and discard'', and we draw as follows;
  $$
  
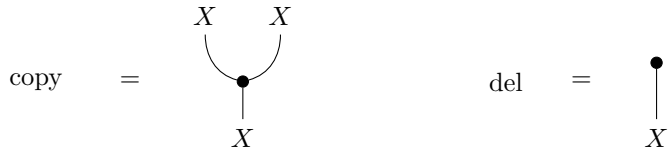
\begin{tikzpicture}
	\begin{pgfonlayer}{nodelayer}
		\node [style=bn] (0) at (-4, 0) {};
		\node [style=none] (1) at (-5, 1.25) {};
		\node [style=none] (2) at (-4, -1) {};
		\node [style=none] (3) at (-3, 1.25) {};
		\node [style=none] (4) at (-5, 1.75) {$X$};
		\node [style=none] (6) at (-3, 1.75) {$X$};
		\node [style=none] (7) at (-4, -1.5) {$X$};
		\node [style=none] (8) at (-9.5, 0) {$\cop$};
		\node [style=bn] (9) at (7, 0.5) {};
		\node [style=none] (10) at (7, -1) {};
		\node [style=none] (11) at (7, -1.5) {$X$};
		\node [style=none] (12) at (3, 0) {$\del$};
		\node [style=none] (13) at (-7, 0) {$=$};
		\node [style=none] (14) at (5, 0) {$=$};
	\end{pgfonlayer}
	\begin{pgfonlayer}{edgelayer}
		\draw [in=165, out=-90] (1.center) to (0);
		\draw [in=270, out=15] (0) to (3.center);
		\draw (0) to (2.center);
		\draw (9) to (10.center);
	\end{pgfonlayer}
\end{tikzpicture}
  $$
  \item The following identities are satisfied (e.g.~``copying and deleting a copy is the same as doing nothing'');
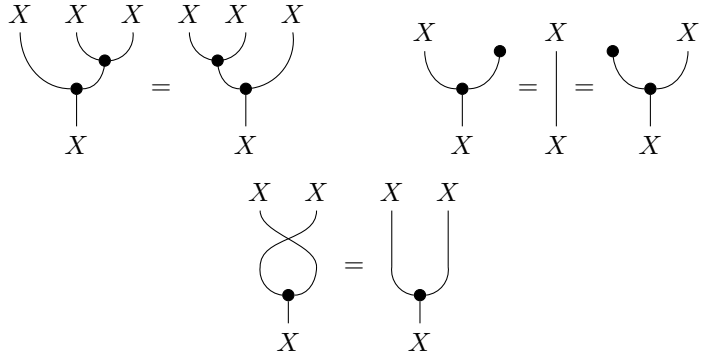
  $$
  \begin{tikzpicture}
	\begin{pgfonlayer}{nodelayer}
		\node [style=bn] (0) at (-2.25, 0) {};
		\node [style=none] (1) at (-3.75, 1.5) {};
		\node [style=none] (2) at (-2.25, -1) {};
		\node [style=none] (3) at (-1.5, 0.75) {};
		\node [style=none] (4) at (-3.75, 2) {$X$};
		\node [style=none] (6) at (-2.25, 2) {$X$};
		\node [style=none] (7) at (-2.25, -1.5) {$X$};
		\node [style=none] (14) at (0, 0) {$=$};
		\node [style=none] (15) at (-2.25, 1.5) {};
		\node [style=none] (16) at (-0.75, 1.5) {};
		\node [style=bn] (17) at (-1.5, 0.75) {};
		\node [style=none] (18) at (-0.75, 2) {$X$};
		\node [style=bn] (19) at (2.25, 0) {};
		\node [style=none] (20) at (1.5, 0.75) {};
		\node [style=none] (21) at (2.25, -1) {};
		\node [style=none] (22) at (1.5, 0.75) {};
		\node [style=none] (23) at (0.75, 2) {$X$};
		\node [style=none] (24) at (2, 2) {$X$};
		\node [style=none] (25) at (2.25, -1.5) {$X$};
		\node [style=none] (26) at (0.75, 1.5) {};
		\node [style=none] (27) at (2.25, 1.5) {};
		\node [style=none] (29) at (3.5, 2) {$X$};
		\node [style=bn] (30) at (1.5, 0.75) {};
		\node [style=none] (31) at (3.5, 1.5) {};
	\end{pgfonlayer}
	\begin{pgfonlayer}{edgelayer}
		\draw [in=-180, out=-90] (1.center) to (0);
		\draw [in=270, out=0] (0) to (3.center);
		\draw (0) to (2.center);
		\draw [bend right=45] (15.center) to (3.center);
		\draw [bend right=45] (17) to (16.center);
		\draw [in=-180, out=-90] (20.center) to (19);
		\draw (19) to (21.center);
		\draw [bend right=45] (26.center) to (22.center);
		\draw [bend left=45] (27.center) to (22.center);
		\draw [bend left=45] (31.center) to (19);
	\end{pgfonlayer}
\end{tikzpicture} \qquad\qquad
  \begin{tikzpicture}
	\begin{pgfonlayer}{nodelayer}
		\node [style=bn] (0) at (-2.5, 0) {};
		\node [style=none] (1) at (-3.5, 1) {};
		\node [style=none] (2) at (-2.5, -1) {};
		\node [style=none] (3) at (-1.5, 1) {};
		\node [style=none] (4) at (-3.5, 1.5) {$X$};
		\node [style=none] (7) at (-2.5, -1.5) {$X$};
		\node [style=none] (14) at (-0.75, 0) {$=$};
		\node [style=none] (28) at (0, 1.5) {$X$};
		\node [style=none] (29) at (0, 1) {};
		\node [style=none] (30) at (0, -1) {};
		\node [style=bn] (31) at (2.5, 0) {};
		\node [style=none] (32) at (1.5, 1) {};
		\node [style=none] (33) at (2.5, -1) {};
		\node [style=none] (34) at (3.5, 1) {};
		\node [style=none] (35) at (0, -1.5) {$X$};
		\node [style=none] (36) at (2.5, -1.5) {$X$};
		\node [style=none] (37) at (3.5, 1.5) {$X$};
		\node [style=none] (38) at (0.75, 0) {$=$};
		\node [style=bn] (39) at (-1.5, 1) {};
		\node [style=bn] (40) at (1.5, 1) {};
	\end{pgfonlayer}
	\begin{pgfonlayer}{edgelayer}
		\draw [in=-180, out=-90] (1.center) to (0);
		\draw [in=270, out=0] (0) to (3.center);
		\draw (0) to (2.center);
		\draw (29.center) to (30.center);
		\draw [in=-180, out=-90] (32.center) to (31);
		\draw [in=270, out=0] (31) to (34.center);
		\draw (31) to (33.center);
	\end{pgfonlayer}
\end{tikzpicture}
  $$
  $$
  \begin{tikzpicture}
	\begin{pgfonlayer}{nodelayer}
		\node [style=bn] (0) at (-1.75, -0.75) {};
		\node [style=none] (1) at (-2.5, 1.5) {};
		\node [style=none] (2) at (-1.75, -1.5) {};
		\node [style=none] (3) at (-1, 0) {};
		\node [style=none] (4) at (-2.5, 2) {$X$};
		\node [style=none] (6) at (-1, 2) {$X$};
		\node [style=none] (7) at (-1.75, -2) {$X$};
		\node [style=none] (14) at (0, 0) {$=$};
		\node [style=none] (15) at (-1, 1.5) {};
		\node [style=bn] (19) at (1.75, -0.75) {};
		\node [style=none] (21) at (1.75, -1.5) {};
		\node [style=none] (25) at (1.75, -2) {$X$};
		\node [style=none] (27) at (1, 1.5) {};
		\node [style=none] (29) at (2.5, 2) {$X$};
		\node [style=none] (31) at (2.5, 1.5) {};
		\node [style=none] (32) at (-2.5, 0) {};
		\node [style=none] (33) at (1, 0) {};
		\node [style=none] (34) at (2.5, 0) {};
		\node [style=none] (35) at (1, 2) {$X$};
	\end{pgfonlayer}
	\begin{pgfonlayer}{edgelayer}
		\draw [in=270, out=0] (0) to (3.center);
		\draw (0) to (2.center);
		\draw (19) to (21.center);
		\draw [bend right=45] (32.center) to (0);
		\draw [bend right=45] (33.center) to (19);
		\draw [bend left=45] (34.center) to (19);
		\draw [in=90, out=-90, looseness=0.75] (1.center) to (3.center);
		\draw [in=90, out=-90] (15.center) to (32.center);
		\draw (27.center) to (33.center);
		\draw (31.center) to (34.center);
	\end{pgfonlayer}
\end{tikzpicture}
  $$
  \item The copy and discard maps are compatible with tensor products in the following way;
  $$
  
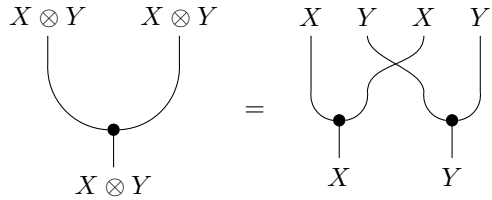
\begin{tikzpicture}
	\begin{pgfonlayer}{nodelayer}
		\node [style=none] (1) at (3, 2) {};
		\node [style=none] (2) at (-3.75, -1.5) {};
		\node [style=none] (3) at (4.5, 0.5) {};
		\node [style=none] (4) at (3, 2.5) {$Y$};
		\node [style=none] (6) at (4.5, 2.5) {$X$};
		\node [style=none] (7) at (-3.75, -2) {$X\otimes Y$};
		\node [style=none] (14) at (0, 0) {$=$};
		\node [style=none] (15) at (4.5, 2) {};
		\node [style=none] (32) at (3, 0.5) {};
		\node [style=none] (33) at (-3.75, -0.5) {};
		\node [style=none] (34) at (-3.75, -1.5) {};
		\node [style=none] (35) at (2.25, -1.25) {};
		\node [style=none] (37) at (2.25, -1.75) {$X$};
		\node [style=none] (38) at (5.25, -1.75) {$Y$};
		\node [style=none] (39) at (2.25, -0.25) {};
		\node [style=none] (40) at (5.25, -0.25) {};
		\node [style=none] (41) at (5.25, -1.25) {};
		\node [style=bn] (42) at (2.25, -0.25) {};
		\node [style=bn] (43) at (5.25, -0.25) {};
		\node [style=none] (44) at (1.5, 0.5) {};
		\node [style=none] (45) at (3, 0.5) {};
		\node [style=none] (46) at (6, 0.5) {};
		\node [style=none] (47) at (1.5, 2) {};
		\node [style=none] (48) at (6, 2) {};
		\node [style=none] (49) at (1.5, 2.5) {$X$};
		\node [style=none] (50) at (6, 2.5) {$Y$};
		\node [style=none] (51) at (-2, 2.5) {$X\otimes Y$};
		\node [style=none] (52) at (-5.5, 2.5) {$X\otimes Y$};
		\node [style=bn] (53) at (-3.75, -0.5) {};
		\node [style=none] (54) at (-5.5, 1.25) {};
		\node [style=none] (55) at (-2, 1.25) {};
		\node [style=none] (56) at (-5.5, 2) {};
		\node [style=none] (57) at (-2, 2) {};
	\end{pgfonlayer}
	\begin{pgfonlayer}{edgelayer}
		\draw [in=90, out=-90, looseness=0.75] (1.center) to (3.center);
		\draw [in=90, out=-90] (15.center) to (32.center);
		\draw (34.center) to (33.center);
		\draw (39.center) to (35.center);
		\draw (40.center) to (41.center);
		\draw [bend right=45] (44.center) to (42);
		\draw [bend left=45] (45.center) to (42);
		\draw [bend right=45] (3.center) to (43);
		\draw [bend left=45] (46.center) to (43);
		\draw (47.center) to (44.center);
		\draw (48.center) to (46.center);
		\draw [bend right=45] (54.center) to (53);
		\draw [bend left=45] (55.center) to (53);
		\draw (56.center) to (54.center);
		\draw (57.center) to (55.center);
	\end{pgfonlayer}
\end{tikzpicture}
  $$
  \item The unit $I$ is the terminal object of the category (i.e.~the category is semicartesian). 
 \end{itemize} 
\end{definition}

In $\cat{Stoch}$, the copy and discard maps are the kernels obtained by the following measurable functions. 
The copy map corresponds from the diagonal map $X\to X\times X$ which literally ``copies the state'', $x\mapsto (x,x)$. The discard map corresponds to the unique map to the one-point space $X\to I$.

\subsection{Graphical definitions of probabilistic concepts}\label{probmarkov}

The formalism of Markov categories allows us to express several concepts of probability theory in categorical and graphical terms. 
The first notion we can look at is stochastic independence. 
First of all, notice that given a joint state $p:I\to X\otimes Y$, we can form the \emph{marginal} $p_X$ on $X$ by simply discarding $Y$.
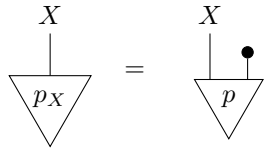
$$
\begin{tikzpicture}
	\begin{pgfonlayer}{nodelayer}
		\node [style=none] (0) at (2.5, -0.75) {};
		\node [style=state] (1) at (2.5, -0.75) {$\; p \;$};
		\node [style=none] (2) at (2, 1) {};
		\node [style=none] (3) at (2, 1.5) {$X$};
		\node [style=none] (4) at (2, -0.5) {};
		\node [style=none] (5) at (3, -0.5) {};
		\node [style=bn] (6) at (3, 0.5) {};
		\node [style=state] (8) at (-2.25, -0.75) {$p_X$};
		\node [style=none] (9) at (-2.25, 1) {};
		\node [style=none] (10) at (-2.25, 1.5) {$X$};
		\node [style=none] (11) at (-2.25, -0.5) {};
		\node [style=none] (12) at (0, 0) {$=$};
	\end{pgfonlayer}
	\begin{pgfonlayer}{edgelayer}
		\draw (2.center) to (4.center);
		\draw (5.center) to (6);
		\draw (9.center) to (11.center);
	\end{pgfonlayer}
\end{tikzpicture}
$$
In $\cat{Stoch}$, this corresponds to saying that the marginal on $X$ is the pushforward of the measure $p$ along the projection map $X\times Y\to X$. 

\begin{definition}
 A joint state $p:I\to X\otimes Y$ is said to exhibit \emph{independence} of $X$ and $Y$ if the following holds. 
 \begin{equation}\label{ind}
 \begin{tikzpicture}
	\begin{pgfonlayer}{nodelayer}
		\node [style=none] (0) at (2.5, -0.25) {};
		\node [style=none] (1) at (5, -0.25) {};
		\node [style=state] (2) at (2.5, -0.25) {$\; p \;$};
		\node [style=state] (3) at (5, -0.25) {$\; p \;$};
		\node [style=none] (4) at (2, 1.5) {};
		\node [style=none] (5) at (5.5, 1.5) {};
		\node [style=none] (6) at (0, 0.25) {$=$};
		\node [style=none] (7) at (-3, 0) {};
		\node [style=none] (8) at (-2, 0) {};
		\node [style=none] (9) at (-2, 1.5) {};
		\node [style=none] (10) at (-3, 1.5) {};
		\node [style=none] (11) at (-3, 2) {$X$};
		\node [style=none] (12) at (-2, 2) {$Y$};
		\node [style=none] (13) at (2, 2) {$X$};
		\node [style=none] (14) at (5.5, 2) {$Y$};
		\node [style=state] (15) at (-2.5, -0.25) {$\; p \;$};
		\node [style=none] (16) at (2, 0) {};
		\node [style=none] (17) at (3, 0) {};
		\node [style=none] (18) at (5.5, 0) {};
		\node [style=none] (19) at (4.5, 0) {};
		\node [style=bn] (20) at (3, 1) {};
		\node [style=bn] (21) at (4.5, 1) {};
	\end{pgfonlayer}
	\begin{pgfonlayer}{edgelayer}
		\draw (18.center) to (5.center);
		\draw (4.center) to (16.center);
		\draw (7.center) to (10.center);
		\draw (8.center) to (9.center);
		\draw (17.center) to (20);
		\draw (19.center) to (21);
	\end{pgfonlayer}
\end{tikzpicture}
 \end{equation}
 
 More generally, a morphism $p:A\to X\otimes Y$ is said to exhibit \emph{conditional independence} of $X$ and $Y$ given $A$ if the following holds.
 $$
 \begin{tikzpicture}
	\begin{pgfonlayer}{nodelayer}
		\node [style=bn] (0) at (3.5, -1.25) {};
		\node [style=none] (1) at (2.5, -0.25) {};
		\node [style=none] (2) at (4.5, -0.25) {};
		\node [style=morphism] (3) at (2.5, 0) {$\; p \;$};
		\node [style=morphism] (4) at (4.5, 0) {$\; p \;$};
		\node [style=none] (5) at (2, 1.25) {};
		\node [style=none] (6) at (5, 1.25) {};
		\node [style=none] (7) at (3.5, -2) {};
		\node [style=none] (8) at (0, 0) {$=$};
		\node [style=none] (9) at (-2.5, -2) {};
		\node [style=none] (10) at (-3, 0) {};
		\node [style=none] (11) at (-2, 0) {};
		\node [style=none] (12) at (-2, 1.25) {};
		\node [style=none] (13) at (-3, 1.25) {};
		\node [style=none] (14) at (-3, 1.75) {$X$};
		\node [style=none] (15) at (-2, 1.75) {$Y$};
		\node [style=none] (16) at (2, 1.75) {$X$};
		\node [style=none] (17) at (5, 1.75) {$Y$};
		\node [style=none] (18) at (3.5, -2.5) {$A$};
		\node [style=none] (19) at (-2.5, -2.5) {$A$};
		\node [style=morphism] (20) at (-2.5, 0) {$\; p \;$};
		\node [style=none] (21) at (2, 0) {};
		\node [style=none] (22) at (3, 0) {};
		\node [style=none] (23) at (5, 0) {};
		\node [style=none] (24) at (4, 0) {};
		\node [style=bn] (25) at (3, 1) {};
		\node [style=bn] (26) at (4, 1) {};
	\end{pgfonlayer}
	\begin{pgfonlayer}{edgelayer}
		\draw (7.center) to (0);
		\draw (23.center) to (6.center);
		\draw (5.center) to (21.center);
		\draw (9.center) to (20);
		\draw (10.center) to (13.center);
		\draw (11.center) to (12.center);
		\draw (22.center) to (25);
		\draw (24.center) to (26);
		\draw [bend right=45] (1.center) to (0);
		\draw [bend right=45] (0) to (2.center);
	\end{pgfonlayer}
\end{tikzpicture}
 $$
\end{definition}
In $\cat{Stoch}$ and $\cat{FinStoch}$ these corresponds to the usual notions. For example, formula \eqref{ind} in $\cat{FinStoch}$ corresponds to saying that 
$$
p(x,y) = p_X(x)\,p_Y(y)
$$
for all $x$ and $y$, and in $\cat{Stoch}$ corresponds to saying that
$$
p(A\times B) = p_X(A)\,p_Y(B)
$$
for all measurable sets $A\in\Sigma_X$ and $B\in\Sigma_Y$.

Another natural concept in Markov categories is the notion of a \emph{deterministic morphism}.

\begin{definition}\label{defdet}
 A  morphism $f:X\to Y$ in a Markov category $\cat{C}$ is called \emph{deterministic} if the following identity holds.
  \begin{equation}\label{det}
   \begin{tikzpicture}
	\begin{pgfonlayer}{nodelayer}
		\node [style=none] (0) at (3, -1.75) {};
		\node [style=bn] (1) at (3, -0.75) {};
		\node [style=none] (2) at (2, 0.25) {};
		\node [style=none] (3) at (4, 0.25) {};
		\node [style=none] (4) at (4, 0.25) {};
		\node [style=morphism] (5) at (2, 0.75) {$f$};
		\node [style=morphism] (6) at (4, 0.75) {$f$};
		\node [style=none] (7) at (-3, -1.75) {};
		\node [style=bn] (8) at (-3, 0.25) {};
		\node [style=none] (9) at (-4, 1.25) {};
		\node [style=none] (10) at (-2, 1.25) {};
		\node [style=none] (11) at (-2, 1.25) {};
		\node [style=none] (12) at (2, 1.75) {};
		\node [style=none] (13) at (4, 1.75) {};
		\node [style=none] (14) at (-4, 1.75) {};
		\node [style=none] (15) at (-2, 1.75) {};
		\node [style=none] (16) at (0, 0) {$=$};
		\node [style=morphism] (17) at (-3, -0.75) {$f$};
		\node [style=none] (18) at (2, 2.25) {$X$};
		\node [style=none] (19) at (4, 2.25) {$X$};
		\node [style=none] (20) at (-4, 2.25) {$X$};
		\node [style=none] (21) at (-2, 2.25) {$X$};
		\node [style=none] (22) at (3, -2.25) {$A$};
		\node [style=none] (23) at (-3, -2.25) {$A$};
	\end{pgfonlayer}
	\begin{pgfonlayer}{edgelayer}
		\draw [style=none] (0.center) to (1);
		\draw [style=none, bend left=45] (1) to (2.center);
		\draw [style=none, bend right=45] (1) to (3.center);
		\draw [style=none] (7.center) to (8);
		\draw [style=none, bend left=45] (8) to (9.center);
		\draw [style=none, bend right=45] (8) to (10.center);
		\draw (12.center) to (5);
		\draw (13.center) to (6);
		\draw (14.center) to (9.center);
		\draw (15.center) to (11.center);
		\draw (2.center) to (5);
		\draw (4.center) to (6);
	\end{pgfonlayer}
\end{tikzpicture} 
  \end{equation}
  We denote by $\cat{C}_\det$ the subcategory of $\cat{C}$ of deterministic morphisms.
\end{definition}
Intuitively, if $f$ carries nontrivial randomness, then the identity above cannot hold: on the left we have perfect correlation given $A$, and on the right we have conditional independence given $A$.

For a state $p:I\to X$, the condition \eqref{det} reads as follows.
\begin{equation}\label{detstate}
   \begin{tikzpicture}
	\begin{pgfonlayer}{nodelayer}
		\node [style=none] (2) at (-3, -0.5) {};
		\node [style=state] (5) at (-3, -1.5) {$p$};
		\node [style=bn] (7) at (-3, -0.5) {};
		\node [style=none] (8) at (-4, 0.5) {};
		\node [style=none] (9) at (-2, 0.5) {};
		\node [style=none] (10) at (-2, 0.5) {};
		\node [style=none] (11) at (-4, 1) {};
		\node [style=none] (12) at (-2, 1) {};
		\node [style=none] (14) at (-4, 1.5) {$X$};
		\node [style=none] (15) at (-2, 1.5) {$X$};
		\node [style=none] (16) at (0, 0) {$=$};
		\node [style=state] (18) at (2, -1.5) {$p$};
		\node [style=none] (20) at (2, -1.5) {};
		\node [style=none] (21) at (4, 0.5) {};
		\node [style=none] (22) at (4, -1.5) {};
		\node [style=none] (23) at (2, 1) {};
		\node [style=none] (24) at (4, 1) {};
		\node [style=none] (25) at (2, 1.5) {$X$};
		\node [style=none] (26) at (4, 1.5) {$X$};
		\node [style=state] (27) at (4, -1.5) {$p$};
	\end{pgfonlayer}
	\begin{pgfonlayer}{edgelayer}
		\draw (5) to (2.center);
		\draw [style=none, bend left=45] (7) to (8.center);
		\draw [style=none, bend right=45] (7) to (9.center);
		\draw (11.center) to (8.center);
		\draw (12.center) to (10.center);
		\draw (23.center) to (20.center);
		\draw (24.center) to (22.center);
	\end{pgfonlayer}
\end{tikzpicture} 
  \end{equation}
In $\cat{Stoch}$ and $\cat{FinStoch}$, these are exactly those measures (and kernels in the general case) which can only have values zero and one. Indeed, formula \eqref{detstate} says that for each pair of measurable sets $A,B\in \Sigma_X$, we have 
$$
p(A\cap B) = p(A)\,p(B) .
$$
In particular, for $A=B$, we get $p(A)=p(A)^2$, i.e.~$p(A)=0$ or $p(A)=1$. 
Every Dirac delta probability measure is of this form. On standard Borel spaces, only Dirac deltas are in this form, and so every deterministic morphism between standard Borel spaces comes from an ordinary measurable function, via the construction \eqref{kernelfrommeasure}. 
If one considers coarser $\sigma$-algebras, however, there are measures which are deterministic (according to \Cref{defdet}), but which are not Dirac deltas. These are extremely important, for example most ergodic measures are of this form, if one considers the $\sigma$-algebra of invariant sets (see \Cref{defergodic}). 
In other words, we have functors
$$
\begin{tikzcd}
 \cat{Meas} \ar{r}{K} & \cat{Stoch}_\det \ar[hookrightarrow]{r} & \cat{Stoch}
\end{tikzcd}
$$
and the first functor is not quite the identity. 
When we speak generally of \emph{deterministic} Markov kernels, we will mean kernels with value zero and one, which are more general than the ones obtained by a measurable function via the functor $K$ of \eqref{kernelfrommeasure}. The advantage of working with this category will become clear in \Cref{ergodic}, and also in \Cref{weakquotient}. 
For more theory on those deterministic morphisms which are not Dirac deltas, see \cite{ours_LICS}.

\begin{proposition}
 The following conditions are equivalent for a Markov category $\cat{C}$:
 \begin{enumerate}
  \item Every morphism of $\cat{C}$ is deterministic.
  \item The copy maps form a natural transformation.
  \item The monoidal structure of $\cat{C}$ is cartesian.
 \end{enumerate}
\end{proposition}
Therefore one can view cartesian categories as precisely those Markov categories with only trivial randomness. 

Another concept of probability theory which can be expressed in terms of Markov categories is the concept of \emph{almost-sure equality}, first introduced in~\cite[Definition~5.1]{chojacobs2019strings} and expanded in~\cite[Definition~13.1]{fritz2019synthetic}.

\begin{definition}\label{as-equal}
 In a Markov category, let $p:A\to X$, and let $f,g:X\to Y$. We say that $f$ and $g$ are \emph{$p$-almost surely equal} if
 $$
 \begin{tikzpicture}
	\begin{pgfonlayer}{nodelayer}
		\node [style=bn] (12) at (-3, 0) {};
		\node [style=none] (14) at (-4, 1) {};
		\node [style=none] (15) at (-4, 2.5) {};
		\node [style=none] (16) at (-2, 2.5) {};
		\node [style=morphism] (19) at (-2, 1.5) {$f$};
		\node [style=none] (20) at (-2, 3) {$Y$};
		\node [style=none] (21) at (-4, 3) {$X$};
		\node [style=none] (22) at (0, 0.25) {$=$};
		\node [style=morphism] (35) at (-3, -1) {$p$};
		\node [style=none] (36) at (-3, -2) {};
		\node [style=none] (37) at (-3, -2.5) {$A$};
		\node [style=bn] (38) at (3, 0) {};
		\node [style=none] (39) at (2, 1) {};
		\node [style=none] (40) at (2, 2.5) {};
		\node [style=none] (41) at (4, 2.5) {};
		\node [style=morphism] (42) at (4, 1.5) {$g$};
		\node [style=none] (43) at (4, 3) {$Y$};
		\node [style=none] (44) at (2, 3) {$X$};
		\node [style=morphism] (45) at (3, -1) {$p$};
		\node [style=none] (46) at (3, -2) {};
		\node [style=none] (47) at (3, -2.5) {$A$};
		\node [style=none] (48) at (-2, 1) {};
		\node [style=none] (49) at (4, 1) {};
	\end{pgfonlayer}
	\begin{pgfonlayer}{edgelayer}
		\draw [bend left=45] (12) to (14.center);
		\draw (14.center) to (15.center);
		\draw (19) to (16.center);
		\draw (35) to (12);
		\draw (36.center) to (35);
		\draw [bend left=45] (38) to (39.center);
		\draw (39.center) to (40.center);
		\draw (42) to (41.center);
		\draw (45) to (38);
		\draw (46.center) to (45);
		\draw [bend left=45] (48.center) to (12);
		\draw [bend left=45] (49.center) to (38);
		\draw (48.center) to (19);
		\draw (49.center) to (42);
	\end{pgfonlayer}
\end{tikzpicture}
 $$
\end{definition} 

In particular, for states this reads as follows. 
 $$
 \begin{tikzpicture}
	\begin{pgfonlayer}{nodelayer}
		\node [style=bn] (12) at (-3, -1) {};
		\node [style=none] (14) at (-4, 0) {};
		\node [style=none] (15) at (-4, 1.5) {};
		\node [style=none] (16) at (-2, 1.5) {};
		\node [style=morphism] (19) at (-2, 0.5) {$f$};
		\node [style=none] (20) at (-2, 2) {$Y$};
		\node [style=none] (21) at (-4, 2) {$X$};
		\node [style=none] (22) at (0, 0) {$=$};
		\node [style=state] (23) at (-3, -2) {$p$};
		\node [style=bn] (26) at (3, -1) {};
		\node [style=none] (28) at (2, 0) {};
		\node [style=none] (29) at (2, 1.5) {};
		\node [style=none] (30) at (4, 1.5) {};
		\node [style=morphism] (31) at (4, 0.5) {$g$};
		\node [style=none] (32) at (4, 2) {$Y$};
		\node [style=none] (33) at (2, 2) {$X$};
		\node [style=state] (34) at (3, -2) {$p$};
		\node [style=none] (35) at (-2, 0) {};
		\node [style=none] (36) at (4, 0) {};
	\end{pgfonlayer}
	\begin{pgfonlayer}{edgelayer}
		\draw [bend left=45] (12) to (14.center);
		\draw (14.center) to (15.center);
		\draw (19) to (16.center);
		\draw (23) to (12);
		\draw [bend left=45] (26) to (28.center);
		\draw (28.center) to (29.center);
		\draw (31) to (30.center);
		\draw (34) to (26);
		\draw (19) to (35.center);
		\draw (31) to (36.center);
		\draw [bend left=45] (36.center) to (26);
		\draw [bend right=45] (12) to (35.center);
	\end{pgfonlayer}
\end{tikzpicture}
 $$
In $\cat{Stoch}$ this condition is equivalent to the usual almost-sure equality for the measure $p$. 

If $f$ is a morphism and $R$ is an equationally-defined property that $f$ may or may not satisfy, we say that $f$ \emph{satisfies the $R$ property $p$-almost surely} if and only if the relevant morphisms are equal $p$-almost surely. 
For example, we say that $f$ is \emph{$p$-almost surely deterministic} if and only if equation \eqref{det} holds $p$-almost surely, i.e.~the following condition holds.
$$
\begin{tikzpicture}
	\begin{pgfonlayer}{nodelayer}
		\node [style=state] (0) at (4, -2) {$p$};
		\node [style=bn] (1) at (4, -1) {};
		\node [style=none] (2) at (2, 1) {};
		\node [style=none] (3) at (5, 0) {};
		\node [style=morphism] (4) at (5, 0.5) {$f$};
		\node [style=bn] (5) at (5, 1.5) {};
		\node [style=none] (6) at (4, 2.5) {};
		\node [style=none] (7) at (6, 2.5) {};
		\node [style=state] (8) at (-4, -2) {$p$};
		\node [style=bn] (9) at (-4, -1) {};
		\node [style=none] (10) at (-6, 1) {};
		\node [style=none] (11) at (-3, 0) {};
		\node [style=bn] (13) at (-3, 0.5) {};
		\node [style=none] (14) at (-4, 1.5) {};
		\node [style=none] (15) at (-2, 1.5) {};
		\node [style=morphism] (16) at (-4, 2) {$f$};
		\node [style=morphism] (17) at (-2, 2) {$f$};
		\node [style=none] (18) at (-4, 3) {};
		\node [style=none] (19) at (-2, 3) {};
		\node [style=none] (20) at (2, 3) {};
		\node [style=none] (21) at (-6, 3) {};
		\node [style=none] (22) at (0, 0) {$=$};
		\node [style=none] (23) at (-6, 3.5) {$X$};
		\node [style=none] (24) at (2, 3.5) {$X$};
		\node [style=none] (25) at (-4, 3.5) {$Y$};
		\node [style=none] (26) at (-2, 3.5) {$Y$};
		\node [style=none] (27) at (4, 3.5) {$Y$};
		\node [style=none] (28) at (6, 3.5) {$Y$};
		\node [style=none] (29) at (4, 3) {};
		\node [style=none] (30) at (6, 3) {};
	\end{pgfonlayer}
	\begin{pgfonlayer}{edgelayer}
		\draw (1) to (0);
		\draw [bend right=45] (2.center) to (1);
		\draw [bend left=45] (3.center) to (1);
		\draw (5) to (4);
		\draw (4) to (3.center);
		\draw [bend right=45] (6.center) to (5);
		\draw [bend left=45] (7.center) to (5);
		\draw (9) to (8);
		\draw [bend right=45] (10.center) to (9);
		\draw [bend left=45] (11.center) to (9);
		\draw [bend right=45] (14.center) to (13);
		\draw [bend left=45] (15.center) to (13);
		\draw (13) to (11.center);
		\draw (18.center) to (16);
		\draw (16) to (14.center);
		\draw (19.center) to (17);
		\draw (17) to (15.center);
		\draw (21.center) to (10.center);
		\draw (20.center) to (2.center);
		\draw (29.center) to (6.center);
		\draw (30.center) to (7.center);
	\end{pgfonlayer}
\end{tikzpicture}
$$
See also \cite[Definition~13.11]{fritz2019synthetic}.

We now turn to conditioning. There are a few variations of the idea of conditionals and disintegrations.
We will use the following one. For additional context, see \cite[Section~11]{fritz2019synthetic}.

\begin{definition}
 In a Markov category $\cat{C}$, let $p:I\to X$ be a state, and let $f:X\to Y$ be a morphism. A \emph{disintegration of $p$} via $f$, or a \emph{Bayesian inversion of $f$ with respect to $p$} is a morphism $f^+_p: Y\to X$ such that the following holds.
 $$
  \begin{tikzpicture}
	\begin{pgfonlayer}{nodelayer}
		\node [style=bn] (13) at (-3, -0.5) {};
		\node [style=none] (14) at (-4, 0.5) {};
		\node [style=morphism] (15) at (-2, 1) {$f$};
		\node [style=none] (16) at (-4, 2) {};
		\node [style=none] (17) at (-2, 2) {};
		\node [style=none] (18) at (-4, 2.5) {$X$};
		\node [style=none] (19) at (-2, 2.5) {$Y$};
		\node [style=morphism] (21) at (3, -1) {$f$};
		\node [style=bn] (22) at (3, 0.5) {};
		\node [style=morphism] (23) at (2, 2) {$f^+_p$};
		\node [style=none] (24) at (4, 1.5) {};
		\node [style=none] (25) at (2, 3) {};
		\node [style=none] (26) at (4, 3) {};
		\node [style=none] (27) at (2, 3.5) {$X$};
		\node [style=none] (28) at (4, 3.5) {$Y$};
		\node [style=none] (29) at (0, 0) {$=$};
		\node [style=none] (30) at (-2, 0.5) {};
		\node [style=none] (31) at (2, 1.5) {};
		\node [style=state] (32) at (-3, -2) {$p$};
		\node [style=state] (33) at (3, -2.5) {$p$};
	\end{pgfonlayer}
	\begin{pgfonlayer}{edgelayer}
		\draw (16.center) to (14.center);
		\draw (17.center) to (15);
		\draw [bend right=45] (14.center) to (13);
		\draw (25.center) to (23);
		\draw (26.center) to (24.center);
		\draw [bend right=45] (22) to (24.center);
		\draw (22) to (21);
		\draw (15) to (30.center);
		\draw (23) to (31.center);
		\draw [bend left=45] (30.center) to (13);
		\draw [bend right=45] (31.center) to (22);
		\draw (13) to (32);
		\draw (33) to (21);
	\end{pgfonlayer}
\end{tikzpicture}
 $$ 
\end{definition}

In $\cat{Stoch}$, this definition reads as follows, if we denote the state $f\circ p:I\to Y$ by $q$. Given a probability measure $p$ on $X$ and a Markov kernel (for example, a measurable function) $f:X\to Y$, the Markov kernel $f^+_p: Y\to X$ is such that for all measurable subsets $A\in\Sigma_X$ and $B\in\Sigma_Y$,
$$
\int_A f(B|x) \,p(dx) = \int_B f^+_p(A|y) \, q(dy) .
$$
In $\cat{FinStoch}$, the condition has the even simpler form
$$
p(x) \, f(y|x) = q(y) \, f^+_p(x|y) .
$$
This can be therefore seen as a categorical definition of Bayesian inversion, and if $f$ is deterministic, as a disintegration of $p$ in the sense of disintegration theorems (see for example \cite[Section~10.6]{bogachev}). 

It follows immediately from the definition that any two disintegrations of $p$ via $f$ as above are equal $q$-almost surely, generalizing what happens in ordinary measure theory.

\begin{proposition}\label{id-as}
 Let $f:X\to Y$ be deterministic. Let $p:I\to X$, and suppose that the disintegration $f^+_p:Y\to X$ exists. Then the composite
 $$
 \begin{tikzcd}
  Y \ar{r}{f^+_p} & X \ar{r}{f} & Y
 \end{tikzcd} 
 $$
 is $(f\circ p)$-almost surely equal to the identity. 
\end{proposition}

This was mentioned in \cite{fritz2019synthetic}, directly after Proposition~11.17 therein.
We include a proof here, for completeness.

\begin{proof}
We have that
 $$
 
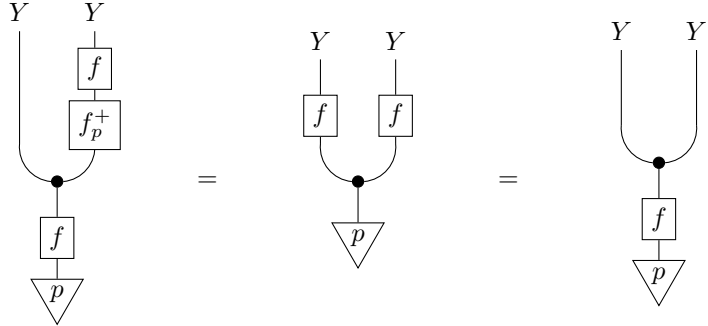
\begin{tikzpicture}
	\begin{pgfonlayer}{nodelayer}
		\node [style=bn] (13) at (0, 0) {};
		\node [style=none] (14) at (-1, 1) {};
		\node [style=none] (16) at (-1, 3.25) {};
		\node [style=none] (17) at (1, 3.25) {};
		\node [style=none] (18) at (-1, 3.75) {$Y$};
		\node [style=none] (19) at (1, 3.75) {$Y$};
		\node [style=morphism] (21) at (-8, -1.5) {$f$};
		\node [style=bn] (22) at (-8, 0) {};
		\node [style=morphism] (23) at (-7, 1.5) {$f^+_p$};
		\node [style=none] (24) at (-9, 1) {};
		\node [style=none] (25) at (-7, 4) {};
		\node [style=none] (26) at (-9, 4) {};
		\node [style=none] (27) at (-7, 4.5) {$Y$};
		\node [style=none] (28) at (-9, 4.5) {$Y$};
		\node [style=none] (29) at (-4, 0) {$=$};
		\node [style=morphism] (30) at (-7, 3) {$f$};
		\node [style=morphism] (34) at (8, -1) {$f$};
		\node [style=bn] (35) at (8, 0.5) {};
		\node [style=none] (37) at (7, 1.5) {};
		\node [style=none] (38) at (9, 3.5) {};
		\node [style=none] (39) at (7, 3.5) {};
		\node [style=none] (40) at (9, 4) {$Y$};
		\node [style=none] (41) at (7, 4) {$Y$};
		\node [style=none] (42) at (9, 1.5) {};
		\node [style=none] (43) at (4, 0) {$=$};
		\node [style=none] (46) at (1, 1) {};
		\node [style=morphism] (47) at (-1, 1.75) {$f$};
		\node [style=morphism] (48) at (1, 1.75) {$f$};
		\node [style=none] (49) at (-7, 1) {};
		\node [style=state] (50) at (-8, -3) {$p$};
		\node [style=state] (51) at (0, -1.5) {$p$};
		\node [style=state] (52) at (8, -2.5) {$p$};
	\end{pgfonlayer}
	\begin{pgfonlayer}{edgelayer}
		\draw [bend right=45] (14.center) to (13);
		\draw (26.center) to (24.center);
		\draw [bend left=45] (22) to (24.center);
		\draw (22) to (21);
		\draw (30) to (23);
		\draw (25.center) to (30);
		\draw (39.center) to (37.center);
		\draw [bend left=45] (35) to (37.center);
		\draw (35) to (34);
		\draw [bend left=45] (42.center) to (35);
		\draw (42.center) to (38.center);
		\draw [bend left=45] (46.center) to (13);
		\draw (16.center) to (47);
		\draw (47) to (14.center);
		\draw (17.center) to (48);
		\draw (48) to (46.center);
		\draw (49.center) to (23);
		\draw [bend left=45] (49.center) to (22);
		\draw (21) to (50);
		\draw (13) to (51);
		\draw (34) to (52);
	\end{pgfonlayer}
\end{tikzpicture}

 $$ 
 where the first equality is by definition of $f^+_p$, and the second equality is by determinism of $f$.
\end{proof}

\begin{definition}
 We say that an object $X$ in a Markov category \emph{has disintegrations} if it admits a disintegration for each state $p:I\to X$ and for each \emph{deterministic} map $f:X\to Y$. 
\end{definition}

In $\cat{Stoch}$, every standard Borel space has disintegrations, this statement is sometimes known as (Rokhlin's) \emph{disintegration theorem}. 
See for example \cite[Theorem~4 and Remark~5]{tao}, \cite[Theorem~5.1.11 and related sections]{viana-oliveira}, as well as \cite[Section~10.6]{bogachev}.

\subsection{Dynamical systems in Markov categories}\label{dynsysmarkov}

A dynamical system in a Markov category can be interpreted as a ``stochastic'' dynamical system in general.
For example, a dynamical system in $\cat{Stoch}$ with monoid $\N$ is a discrete-time Markov process. 
In this work we are mostly interested in dynamical systems in the subcategory of deterministic morphisms of a Markov category. These are interpretable as traditional deterministic dynamical systems. The advantage of working in the larger Markov category (rather than in $\cat{Meas}$) is the convenience of having states (measures) and conditionals (kernels) fit in the same language. 
Moreover, as we have seen, there are deterministic morphisms in $\cat{Stoch}$ which are not just measurable functions, and these are going to be crucial to talk about ergodicity. 

Let $X$ be a dynamical system with monoid $M$ in a Markov category $\cat{C}$. 
Following the intuition of \Cref{dynsyscat}, we have the following.
\begin{itemize}
 \item A left-invariant state (a.k.a.~cone) from the monoidal unit $p:I\to X$ is a state satisfying $m\circ p=p$ for all $m\in M$. This can be interpreted as an ``invariant random state'', or ``invariant measure''. 
 In $\cat{Stoch}$, these are invariant measures. For deterministic dynamical systems generated by measurable functions in the form $m:X\to X$, this is a measure $p$ satisfying
 $$
 p(m^{-1}(A)) = p(A) 
 $$
 for every measurable set $A\in\Sigma_X$. More generally, for dynamical systems generated by kernels, in the form $m:X\times\Sigma_X\to[0,1]$, the invariance of $p$ means that for every measurable $A\in\Sigma_X$,
 $$
 \int_X m(A|x) \,p(dx) = p(A) .
 $$
 
 \item More generally, a left-invariant morphism (a.k.a.~cone) from a generic object $c:C\to X$ is a morphism satisfying $m\circ c=c$ for all $m\in M$. This can be interpreted either as a ``transition to an invariant state'', or as a family of invariant states parametrized (measurably) by $C$. 
 In $\cat{Stoch}$, the interpretation is similar to invariant measures, except that they depend measurably on a parameter. 
 
 \item A \emph{right}-invariant morphism (a.k.a.~\emph{co}cone) $r:X\to R$ is a morphism satisfying $r\circ m=r$ for all $m\in M$. This can be interpreted as an ``invariant function or invariant observable'', especially when $r$ is deterministic. 
 In $\cat{Stoch}$, these indeed correspond to invariant functions, or invariant kernels. For dynamical systems where $M$ acts by measurable functions $m:X\to X$, this means a measurable map $r:X\to R$ satisfying
 $$
 r(m(x)) = r(x) 
 $$
 for every $x\in X$ and $m \in M$. More generally, for dynamical systems where $M$ acts by kernels $m:X\times\Sigma_X\to[0,1]$, the right-invariance of a kernel $r$ means that for every $x\in X$, $m\in M$ and $B\in\Sigma_R$,
 \begin{equation}\label{rinv}
 \int_X r(B|x') \, m(dx'|x) = r(B|x) .
 \end{equation}
\end{itemize}

From now on, when we talk about an \emph{invariant state}, we always mean a \emph{left}-invariant state (in $\cat{Stoch}$, an invariant measure). When we talk about an \emph{invariant observable}, we talk about a \emph{right}-invariant morphism (in $\cat{Stoch}$, an invariant measurable map or kernel). We are mostly interested in \emph{deterministic} invariant observables.

\section{Ergodic decomposition in Markov categories}\label{ergodic}

In order to express our main result (\Cref{mainthm}), we need to express some additional concepts in terms of category theory. 
First of all, we need a notion of ``mixture'' or ``convex combination'', which will be introduced in \Cref{mixtures}.
We then need to explain how to talk about invariant sets categorically, and it will be in terms of a particular colimit construction, explained in \Cref{markovquotient}.
Finally, in \Cref{ergodicstates} we define ergodic states as particular deterministic morphisms, which allows in \Cref{mainsec} to express our main result and its instantiation in the category of Markov kernels (\Cref{mainstoch}).

\subsection{Mixtures of states}\label{mixtures}

Let's define a categorical version of convex decompositions, which can be constructed in any Markov category. 
 
For motivation, let $X$ be a finite set and let $p$ be a (discrete) probability measure on $X$. Let now $k_1,\dots,k_n$ be other probability measures on $X$. We say that $p$ is a convex combination of the $k_i$ with coefficients $q_i$ if 
$$
 p = \sum_{i=1}^n k_i \, q_i ,
$$
or more explicitly, for each $x\in X$,
\begin{equation}\label{convcomb}
 p(x) = \sum_{i=1}^n k_i(x) \, q_i .
\end{equation}

In order for the $q_i$ to be the coefficients of a convex combination, we need $0\le q_i\le 1$ for all $i$, and $\sum_i q_i=1$. 
In other words, we need the map $i\mapsto q_i$ to be a (discrete) probability measure on the set $\{1,\dots,n\}$. 
That way, $i\mapsto k_i$ can be considered a discrete kernel (or transition matrix) from $\{1,\dots,n\}$ to $X$. 
The equation \eqref{convcomb} can then be expressed as the (matrix) composition of $q$ with $k$. 
Note that the values of $k_i$ when $q_i=0$ do not play any role.

More generally, if a generic finite set $Y$ is indexing the convex combination, and writing $k(x|y)$ instead of $k_y(x)$, we see that a convex decomposition of $p$,
$$
 p(x) = \sum_{y\in Y} k(x|y) \, q(y) ,
 $$
is just the decomposition of $p:I\to X$ in the category $\cat{FinStoch}$ into a composite of stochastic matrices, 
$$
\begin{tikzcd}
 I \ar{r}{q} & Y \ar{r}{k} & X .
\end{tikzcd}
$$

Let's now turn to the continuous case. Let $X$ be a measurable space, and let $p$ be a probability measure on $X$, which we can view as a morphism $p:I\to X$ of $\cat{Stoch}$.
In this category, decomposing $p:I\to X$ corresponds to writing it as a (measurably-indexed) mixture of measures. Indeed, $k\circ q = p$ for a pair of morphisms $q:I\to Y, k:Y\to X$ if and only if for each measurable set $A\subseteq X$, we have
 $$
 p(A) = \int_Y k(A|y) \, q(dy) .
 $$
 If we consider $k$ as a $Y$-indexed family of measures on $X$, and denote it by $y\mapsto k_y$, then we are equivalently saying that
 $$
 p = \int_Y k_y \, q(dy) .
 $$
 That is, $p$ is a mixture of measures ($k_y$) with mixing measure $q$. 
Note that the mixture only depends on the values of $k_y$ for $q$-almost all $y$.
 
Here is the general definition. 
\begin{definition}
 Let $p:I\to X$ be a state in a Markov category. 
 A \emph{decomposition} of $p$ is a factorization of $p$.
 More explicitly, it consists of 
 \begin{itemize}
  \item a state $q:I\to Y$,
  \item a morphism $k : Y \to X$
 \end{itemize}
 such that $k\circ q = p$.
\end{definition}
Not all distinctions between decompositions are interesting.
For example, in the discussion above, changes to the values of $k_y$ on any set of $q$-measure zero are not interesting.
In terms of Markov categories, we have the following.
If $(q,k)$ is a decomposition of $p$ and $k' : Y \to X$ is any other morphism with $k \ase q k'$, it follows that $p = k \circ q = k' \circ q$, whence $(q,k')$ is also a decomposition of $p$.
Thus it would be somewhat natural to identify decompositions $(q,k), (q,k')$ whenever $k \ase q k'$, but equivalence of decompositions plays no role in this paper beyond the special case of equivalence with the \emph{trivial} decomposition (\Cref{trivialdecomposition}).

In the discrete case, $p$ always has a trivial decomposition: we can write $p$ as
$$
p = \sum_{y\in Y} p \, q(y) ,
$$
for any normalized measure $q$. More generally, we can always write
$$
p = \int_Y \tilde{p}_y \, q(dy) ,
$$
where $y\mapsto \tilde{p}_y$ is $q$-almost surely equal to $p$. 

We can define this in general. 

\begin{definition}\label{trivialdecomposition}
 Let $(q:I\to Y, k:Y\to X)$ be a decomposition of $p:I\to X$. 
 We say that $(q,k)$ is a \emph{trivial} decomposition of $p$ if and only if $k$ is $q$-almost surely equal to 
 $$
 \begin{tikzcd}
  Y \ar{r}{\del} & I \ar{r}{p} & X . 
 \end{tikzcd}
 $$
 
 We call the state $p$ \emph{indecomposable} if all its decompositions are trivial. 
\end{definition}

Clearly, each Dirac delta probability measure is indecomposable. 
This is part of Choquet theory, where one shows that the set of probability measures over a given space is a simplex (or an infinite-dimensional analogue thereof), and its extreme points are exactly the Dirac delta, or more generally, the zero-one measures. (See for example \cite{winkler} for more on Choquet decompositions.)
For general Markov categories, there is a similar relationship between indecomposable and deterministic states. 

\begin{proposition}
 Every indecomposable state is deterministic. 
\end{proposition}

\begin{proof}
 Let $p:I\to X$ be an indecomposable state. We can decompose $p$ as $p=\id\circ p$, and by hypothesis this decomposition is trivial. Therefore, the identity $\id$ is $p$-almost surely equal to 
 $$
 \begin{tikzcd}
  X \ar{r}{\del} & I \ar{r}{p} & X . 
 \end{tikzcd}
 $$
 That is, 
 $$
 
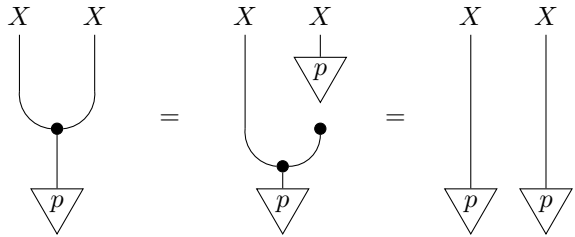
\begin{tikzpicture}
	\begin{pgfonlayer}{nodelayer}
		\node [style=bn] (12) at (-6, 0) {};
		\node [style=none] (14) at (-7, 1) {};
		\node [style=none] (15) at (-7, 2.5) {};
		\node [style=none] (16) at (-5, 2.5) {};
		\node [style=none] (19) at (-5, 1) {};
		\node [style=none] (20) at (-5, 3) {$X$};
		\node [style=none] (21) at (-7, 3) {$X$};
		\node [style=none] (22) at (-3, 0.25) {$=$};
		\node [style=state] (23) at (-6, -2) {$p$};
		\node [style=bn] (26) at (0, -1) {};
		\node [style=none] (28) at (-1, 0) {};
		\node [style=none] (29) at (-1, 2.5) {};
		\node [style=none] (30) at (1, 2.5) {};
		\node [style=bn] (31) at (1, 0) {};
		\node [style=none] (32) at (1, 3) {$X$};
		\node [style=none] (33) at (-1, 3) {$X$};
		\node [style=state] (34) at (0, -2) {$p$};
		\node [style=none] (35) at (3, 0.25) {$=$};
		\node [style=state] (36) at (1, 1.5) {$p$};
		\node [style=none] (39) at (5, 2.5) {};
		\node [style=none] (40) at (5, 3) {$X$};
		\node [style=state] (41) at (5, -2) {$p$};
		\node [style=none] (42) at (7, 2.5) {};
		\node [style=none] (43) at (7, 3) {$X$};
		\node [style=state] (44) at (7, -2) {$p$};
	\end{pgfonlayer}
	\begin{pgfonlayer}{edgelayer}
		\draw [bend right=315] (12) to (14.center);
		\draw (14.center) to (15.center);
		\draw (19.center) to (16.center);
		\draw [bend right=45] (12) to (19.center);
		\draw (23) to (12);
		\draw [bend left=45] (26) to (28.center);
		\draw (28.center) to (29.center);
		\draw [bend right=45] (26) to (31);
		\draw (34) to (26);
		\draw (30.center) to (36);
		\draw (39.center) to (41);
		\draw (42.center) to (44);
	\end{pgfonlayer}
\end{tikzpicture}
 $$
 Hence $p$ is deterministic. 
\end{proof}
 
The converse statement fails for general Markov categories, but it holds for a large class of them, including $\cat{Stoch}$ and $\cat{FinStoch}$. 
Recall from \cite[Section~11]{fritz2019synthetic} that a Markov category is called \emph{positive} if whenever a composite $f\circ p$ is deterministic, then the following equation holds. 
\begin{equation}\label{pos}
\begin{tikzpicture}
	\begin{pgfonlayer}{nodelayer}
		\node [style=none] (0) at (-3, -2) {};
		\node [style=bn] (1) at (-3, 0) {};
		\node [style=none] (2) at (-2, 1) {};
		\node [style=none] (3) at (-4, 1) {};
		\node [style=none] (4) at (-4, 1) {};
		\node [style=morphism] (5) at (-2, 1.5) {$f$};
		\node [style=none] (6) at (-4, 1.25) {};
		\node [style=none] (7) at (3, -2) {};
		\node [style=bn] (8) at (3, -1) {};
		\node [style=none] (9) at (4, 0) {};
		\node [style=none] (10) at (2, 0) {};
		\node [style=none] (11) at (2, 0) {};
		\node [style=none] (12) at (-2, 2.5) {};
		\node [style=none] (13) at (-4, 2.5) {};
		\node [style=none] (14) at (4, 2.5) {};
		\node [style=none] (15) at (2, 2.5) {};
		\node [style=none] (16) at (0, 0) {$=$};
		\node [style=morphism] (17) at (-3, -1) {$p$};
		\node [style=none] (18) at (-2, 3) {$Y$};
		\node [style=none] (19) at (-4, 3) {$X$};
		\node [style=none] (20) at (4, 3) {$Y$};
		\node [style=none] (21) at (2, 3) {$X$};
		\node [style=none] (22) at (-3, -2.5) {$A$};
		\node [style=none] (23) at (3, -2.5) {$A$};
		\node [style=morphism] (24) at (4, 0.25) {$p$};
		\node [style=morphism] (25) at (2, 0.25) {$p$};
		\node [style=morphism] (26) at (4, 1.5) {$f$};
	\end{pgfonlayer}
	\begin{pgfonlayer}{edgelayer}
		\draw [style=none] (0.center) to (1);
		\draw [style=none, bend right=45] (1) to (2.center);
		\draw [style=none, bend left=45] (1) to (3.center);
		\draw [style=none] (7.center) to (8);
		\draw [style=none, bend right=45] (8) to (9.center);
		\draw [style=none, bend left=45] (8) to (10.center);
		\draw (12.center) to (5);
		\draw (13.center) to (6.center);
		\draw (14.center) to (9.center);
		\draw (15.center) to (11.center);
		\draw (2.center) to (5);
		\draw (4.center) to (6.center);
	\end{pgfonlayer}
\end{tikzpicture}
\end{equation}
This holds for $\cat{Stoch}$ and $\cat{FinStoch}$, and it is related to the fact that probabilities are nonnegative (hence the name), see the original reference for more details.

\begin{proposition}
 In a positive Markov category, every deterministic state is indecomposable. 
\end{proposition}

\begin{proof}
 If we instantiate the positivity condition \eqref{pos} with the case of a state we get that for $p:I\to X$ and $f:X\to Y$ such that $f\circ p$ is deterministic, then
 \begin{equation}\label{pos-states}
 \begin{tikzpicture}
	\begin{pgfonlayer}{nodelayer}
		\node [style=bn] (1) at (-6, -0.75) {};
		\node [style=none] (2) at (-5, 0.25) {};
		\node [style=none] (3) at (-7, 0.25) {};
		\node [style=none] (4) at (-7, 0.25) {};
		\node [style=morphism] (5) at (-5, 0.75) {$f$};
		\node [style=none] (6) at (-7, 0.5) {};
		\node [style=none] (12) at (-5, 1.75) {};
		\node [style=none] (13) at (-7, 1.75) {};
		\node [style=none] (14) at (1, 1.75) {};
		\node [style=none] (15) at (-1, 1.75) {};
		\node [style=none] (16) at (-3, 0) {$=$};
		\node [style=none] (18) at (-5, 2.25) {$Y$};
		\node [style=none] (19) at (-7, 2.25) {$X$};
		\node [style=none] (20) at (1, 2.25) {$Y$};
		\node [style=none] (21) at (-1, 2.25) {$X$};
		\node [style=morphism] (26) at (1, 0) {$f$};
		\node [style=state] (27) at (-6, -1.75) {$p$};
		\node [style=state] (28) at (-1, -1.75) {$p$};
		\node [style=state] (29) at (1, -1.75) {$p$};
		\node [style=none] (30) at (3, 0) {$=$};
		\node [style=bn] (31) at (6, -2.25) {};
		\node [style=none] (33) at (5, -1.25) {};
		\node [style=none] (34) at (5, -1.25) {};
		\node [style=none] (36) at (5, -1) {};
		\node [style=none] (38) at (5, 2.75) {};
		\node [style=none] (40) at (5, 3.25) {$X$};
		\node [style=state] (41) at (6, -3.25) {$p$};
		\node [style=bn] (42) at (7, -1.25) {};
		\node [style=none] (43) at (7, 2.75) {};
		\node [style=none] (44) at (7, 3.25) {$Y$};
		\node [style=morphism] (45) at (7, 1.75) {$f$};
		\node [style=state] (46) at (7, 0.25) {$p$};
	\end{pgfonlayer}
	\begin{pgfonlayer}{edgelayer}
		\draw [style=none, bend right=45] (1) to (2.center);
		\draw [style=none, bend left=45] (1) to (3.center);
		\draw (12.center) to (5);
		\draw (13.center) to (6.center);
		\draw (2.center) to (5);
		\draw (4.center) to (6.center);
		\draw (1) to (27);
		\draw (15.center) to (28);
		\draw (26) to (29);
		\draw (14.center) to (26);
		\draw [style=none, bend left=45] (31) to (33.center);
		\draw (38.center) to (36.center);
		\draw (34.center) to (36.center);
		\draw (31) to (41);
		\draw [bend left=45] (42) to (31);
		\draw (45) to (46);
		\draw (43.center) to (45);
	\end{pgfonlayer}
\end{tikzpicture}
 \end{equation}
 So let the composition $f\circ p$ be a deterministic state. The equation above says that $f$ is $p$-almost surely equal to $f\circ p$. Therefore $f\circ p$ is indecomposable.
\end{proof}

In conclusion,
\begin{itemize}
 \item In $\cat{Stoch}$ and $\cat{FinStoch}$, convex combinations of measures can be described categorically as compositions of arrows (up to almost sure equality);
 \item The deterministic states are precisely those that cannot be written as a nontrivial convex combination.
\end{itemize}
This is the notion of decomposition which we will use in our ergodic decomposition theorem (\Cref{mainsec}).

\subsection{Markov quotients}\label{markovquotient}

For the purposes of this work, we need \emph{colimits} of dynamical systems in a Markov category, which as we have seen in \Cref{dynsyscat}, can be often interpreted as ``quotient spaces'' or ``spaces of orbits'' (more on this in \Cref{weakquotient}). 
In Markov categories we need to require a little bit more of the usual universal property: it needs to play well with deterministic morphisms.

\begin{definition}
 Let $X$ be a dynamical system with monoid $M$ in a Markov category $\cat{C}$. 
 The \emph{Markov colimit} or \emph{Markov quotient} of $X$ over the action of $M$ is an object, which we denote by $X_\inv$ together with a deterministic map $r:X\to X_\inv$, which is a colimit both in $\cat{C}$ and in the subcategory $\cat{C}_\det$ of deterministic morphisms.
\end{definition}

Alternatively and more explicitly, it is a map $r : X \to X_\inv$ where
\begin{itemize}
\item for every invariant observable $s:X\to S$ there exists a unique map $\tilde s:X_\inv\to S$ such that $s=\tilde s\circ r$, i.e.~making the following diagram commute for each $m\in M$;
  \begin{equation}\label{defXinv}
    \begin{tikzcd}[row sep=small]
      X \ar{dd}[swap]{m} \ar{dr}[swap]{r} \ar[bend left=10]{drr}{s} \\
      & X_\inv \ar[dashed]{r}[near start]{\tilde s} & S \\
      X \ar{ur}{r} \ar[bend right=10]{urr}[swap]{s}
    \end{tikzcd}
  \end{equation} 
\item and moreover the map $\tilde s$ is deterministic if and only if the map $s$ is.
\end{itemize}
By taking $S = X_\inv$ and $s = r$, so that $\tilde s = \id$, we see that the second point implies that $r$ is deterministic.

It is worth remarking on the connection between our Markov quotients and the \emph{Kolmogorov products} introduced in \cite{fritzrischel2019zeroone}.
The latter essentially suggests a notion of cofiltered limit appropriate to Markov categories, characterized by requiring the limiting property to hold in both the $\cat{C}$ and $\cat{C}_\det$ and to be preserved by all functors $(-) \otimes A$ for $A \in \cat{C}$.
By contrast, we do not require Markov quotients to be preserved by the tensor product here.

We have seen that the colimit of a dynamical system has the interpretation of a ``space of orbits''. 
In $\cat{Stoch}$, at least when the dynamics is deterministic, the Markov colimit exists and is given by the \emph{invariant $\sigma$-algebra}. 
We'll define this more precisely and then check the universal property (\Cref{invsigmacolimit}). In \Cref{weakquotient} we will see in what sense it is similar to a space of orbits. 

\begin{definition}\label{definvset}
 Let $X$ be a dynamical system in $\cat{Stoch}$ with monoid $M$. A measurable set $A\in\Sigma_X$ is called \emph{invariant} if for every $m\in M$ we have 
 \begin{equation}\label{invset}
 m(A|x) = 1_A(x) = 
 \begin{cases}
  1 & x\in A ; \\
  0 & x\notin A .
 \end{cases}
 \end{equation}
\end{definition}

If the dynamical system is deterministic and generated by measurable functions $m:X\to X$, the invariance condition \eqref{invset} can be written more simply as $m^{-1}(A)=A$. 
As it is well known, both for the deterministic and for the generic case, invariant sets form a $\sigma$-algebra, often called the \emph{invariant $\sigma$-algebra}.

\begin{proposition}\label{invsigmacolimit}
 Let $X$ be a \emph{deterministic} dynamical system in $\cat{Stoch}$ with monoid $M$. 
 Then the Markov quotient $X_\inv$ exists, and it is given by the same set $X$, equipped with the invariant $\sigma$-algebra. 
\end{proposition}

\begin{proof}
 Construct the kernel $r:X\to X_\inv$ as follows,
 $$
 r(A|x) = 1_A(x) = \begin{cases} 
                                   1 & x\in A \\
                                   0 & x\notin A
                                  \end{cases}
 $$
 for every $x\in S$ and every measurable (invariant) set $A\in \Sigma_{X_\inv}$.
 Note that it is exactly the kernel induced by the function $X\to X_\inv$ induced by the set-theoretic identity. As every measurable set of $X_\inv$ is measurable in $X$, this function is measurable, and so it induces a well-defined Markov kernel.
 Let's now prove that $r$ is left-invariant. For every $m\in M$, every $x\in X$ and every measurable (invariant) set $A\in \Sigma_{X_\inv}$, 
 $$
  \int_X r(A|x') \,m(dx'|x) = \int_X 1_A(x') \,m(dx'|x) = m(A|x) = 1_A(x) = r(A|x),
 $$
 where we used invariance of $A$. 
 
 Let's now prove the universal property \eqref{defXinv}.
 Let $s:X\to S$ be a right-invariant Markov kernel. 
 Define the kernel $\tilde s:X_\inv\to S$ simply by 
 $$
 \tilde s(B|x) \coloneqq s(B|x) 
 $$
 for all $x\in X_\inv$ (equivalently, $x\in X$) and all measurable $B\subseteq S$. 
 To see that $\tilde s$ is measurable in $x$, consider a Borel-generating interval $(r,1]\subseteq [0,1]$ for some $0\le r < 1$. We have to prove that the set 
 \begin{align*}
 \tilde s^*(B,r) &\coloneqq   \{x\in X_\inv : s(B|x) > r\}
 \end{align*}
 is measurable in $X_\inv$, i.e., as a subset of $X$, is measurable and invariant.
 We know that it is measurable as a subset of $X$, since $s$ is a Markov kernel. Let's prove invariance.
 Using the fact that $s$ is right-invariant (equation~\eqref{rinv}), and that $\tilde s^*(B,r)$ is measurable as a subset of $X$,
 \begin{align*}
 s(B|x) &= \int_X s(B|x') \, m(dx'|x) \\
 &= \int_{\tilde s^*(B,r)} s(B|x') \, m(dx'|x) + \int_{X\setminus \tilde s^*(B,r)} s(B|x') \, m(dx'|x)
 \end{align*}
 
 Now let's use that $m$ is deterministic, so that either we have $m(\tilde s^*(B,r)|x)=1$, or $m(X\setminus \tilde s^*(B,r)|x)=1$. 
 In the first case, we have that $s(B|x') > r$ on a set of measure $1$, therefore $s(B|x)> r$, i.e.~$x\in \tilde s^*(B,r)$. 
 In the latter case, $s(B|x') \le r$ on a set of measure $1$, and therefore $s(B|x) \le r$, i.e.~$x\notin \tilde s^*(B,r)$. 
 Since $m$ is deterministic, these are the only possibilities, and so we have that 
 $$
 m(\tilde s^*(B,r)|x) = \begin{cases}
                       1 & x \in \tilde s^*(B,r) \\
                       0 & x\notin \tilde s^*(B,r) .
                      \end{cases}
 $$ 
 This means precisely that $\tilde s^*(B,r)$ is invariant, and so $\tilde s$ is measurable. 
 Therefore, $\tilde s$ is a well-defined Markov kernel $X_\inv\to S$.
  
 For uniqueness, note that $\tilde s$ is the only possible choice of kernel $X_\inv\to S$ making \eqref{defXinv} commute: let $k:X_\inv\to S$ be another such kernel. Then for all $x\in X$ and every measurable $B\subseteq X$,
 $$
 k(B|x) = \int_{X_\inv} k(B|x') \,\delta(dx'|x) = \int_{X_\inv} k(B|x') \, r(dx'|x) =  s(B|x).
 $$
 
 Moreover, by construction, $\tilde s$ is deterministic if and only if $s$ is.
\end{proof}

\subsection{Ergodic states}\label{ergodicstates}

\begin{definition}\label{defergodic}
 Let $X$ be a dynamical system with monoid $M$ in a Markov category $\cat{C}$. 
 An invariant state $p:I\to X$ is \emph{ergodic} if for every invariant deterministic observable $c:X\to R$, the composition $c\circ p$ is deterministic. 
\end{definition}

Intuitively, an ergodic measure is an invariant state of the system such that every conserved quantity almost surely takes on a single definite value.
In particular, invariant deterministic observations $c : X \to R$ cannot be used to `decompose' the state $p$ into disjoint invariant subsystems of $X$, since $c \circ p$ being deterministic intuitively means that the state $p$ is concentrated in a single fibre of $c$.

\begin{example}
 Let $X$ be a standard Borel space, and consider the infinite product $X^\N$ with the product $\sigma$-algebra. 
 The group $S_\N$ of finite, unbounded permutations of $\N$ acts on $X^\N$ by permuting the components.
 Let now $p$ be a measure on $X$. The infinite product measure $p^{\otimes\N}$ on $X^\N$ is clearly permutation-invariant.
 The \emph{Hewitt-Savage zero-one law}~\cite{hewitt-savage} says that for every permutation-invariant deterministic observable $c:X^\N\to R$, the pushforward of $p^{\otimes\N}$ along $c$ is a zero-one measure. 
 As a zero-one measure is exactly a deterministic state according to our definition, the Hewitt-Savage zero-one law can be equivalently expressed as follows: for the dynamical system $X^\N$ with the dynamics given by $S_\N$, any infinite product measure is ergodic. 
 
 A categorical proof of the Hewitt-Savage zero-one law, in terms of Markov categories, has been given in \cite{fritzrischel2019zeroone}. 
\end{example}

Our notion of ergodicity coincides with the traditional one in terms of invariant sets, by means of the following statement, which follows directly from the universal property of Markov colimits.

\begin{proposition}\label{erguni}
 Let $X$ be a dynamical system with monoid $M$ in a Markov category $\cat{C}$, and suppose that the Markov colimit $X_\inv$ of $X$ exists.
 An invariant state $p:I\to X$ is ergodic if and only if the composition with the universal cocone
 $$
 \begin{tikzcd}
  I \ar{r}{p} & X \ar{r}{r} & X_\inv
 \end{tikzcd}
 $$
 is deterministic. 
\end{proposition}

\begin{proof}
  First, suppose that the composite $r\circ p$ is deterministic. Let $c:X\to R$ be an invariant deterministic observable. By definition of Markov colimit, $c$ factors (uniquely) as a composite $\tilde{c}\circ r$, where $\tilde{c}:X_\inv\to R$ is deterministic. Therefore 
 \[
  c\circ p = \tilde{c}\circ r \circ p =  \tilde{c}\circ (r\circ p)
 \]
 is a composite of deterministic maps, and hence is deterministic. This is true for every invariant deterministic $c$, and so $p$ is ergodic.
 
 The converse follows by taking $c$ in the definition of ergodicity to be $r:X\to X_\inv$, which is deterministic and invariant. 
\end{proof}

\begin{corollary}\label{cor-erg-in-stoch}
 Let $X$ be a deterministic dynamical system in $\cat{Stoch}$. An invariant measure $p$ on $X$ is ergodic (according to our \Cref{defergodic}) if and only if for every invariant measurable set $A$, we have $p(A)=0$ or $p(A)=1$ (i.e.~$p$ is ergodic in the traditional sense).
\end{corollary}

Compare for example with the traditional characterizations \cite[Theorem~3]{tao} and \cite[Proposition~4.1.3]{viana-oliveira}.

In order to state \Cref{mainthm}, it remains to consider \emph{families} of states which are almost surely ergodic.
In \Cref{defergodic} we required $p : I \to X$ to be an invariant state, but the definition still makes sense for an invariant morphism $k : Y \to X$ (i.e.\ with domain other than $I$).
The development above generalizes immediately and in particular, in $\cat{Stoch}$, $k : Y \to X$ being ergodic simply says that each measure $k_y$ on $X$ for $y \in Y$ is ergodic.
However, we cannot simply derive the meaning of `almost surely ergodic' in the manner described after \Cref{as-equal}, since ergodicity is not a purely equational notion.
Nevertheless it is natural to adopt the following definition.

\begin{definition}\label{def-as-erg}
  Let $X$ be a dynamical system with monoid $M$ in a Markov category $\cat{C}$.
  Let $Y$ be an object of $\cat{C}$ with a state $q : I \to Y$ and a morphism $k : Y \to X$.
  We say that $k$ is \emph{$q$-almost surely ergodic} if
  \begin{itemize}
  \item $k$ is $q$-almost surely invariant, and
  \item whenever $r : X \to R$ is invariant and deterministic (not just almost surely), then $r \circ k$ is $q$-almost surely deterministic.
  \end{itemize}
\end{definition}

The following is a straightforward adaptation of \Cref{erguni}.

\begin{proposition}\label{ergasuni}
  Let $X$ be a dynamical system with monoid $M$ in a Markov category $\cat{C}$, and suppose that the Markov colimit $X_\inv$ of $X$ exists.
  Let $Y$ be an object of $\cat{C}$ with a state $q : I \to Y$ and a $q$-almost surely invariant morphism $k : Y \to X$.
  Then $k$ is $q$-almost surely ergodic if and only if the composition with the universal cocone
  $$
  \begin{tikzcd}
    Y \ar{r}{k} & X \ar{r}{r} & X_\inv
  \end{tikzcd}
  $$
  is $q$-almost surely deterministic.
\end{proposition}

\Cref{def-as-erg} is justified by what it means for $\cat{Stoch}$.

\begin{corollary}
  Let $X$ be a deterministic dynamical system in $\cat{Stoch}$ and let $q : I \to Y$ be a measure on some measurable space $Y$.
  A stochastic map $k : Y \to X$ on $X$ is $q$-almost surely ergodic (according to our \Cref{def-as-erg}) if and only if there is a measurable set $E \subseteq Y$ with $q(E) = 1$ such that $k_y$ is ergodic (in the usual sense) for each $y \in E$.
\end{corollary}
\begin{proof}
  By \Cref{invsigmacolimit} the Markov quotient $r : X \to X_\inv$ exists.
  Suppose $k$ satisfies the second condition, which in particular certainly means that $k$ is $q$-almost surely invariant.
  Using \Cref{cor-erg-in-stoch} we see that $r \circ k$ is $q$-almost surely deterministic and hence, by \Cref{ergasuni}, $k$ is $q$-almost surely ergodic.
  For the converse, we can take a set $E_1 \subseteq Y$ with $q(E_1) = 1$ and $p_y$ invariant for $y \in E_1$, and a set $E_2 \subseteq Y$ with $q(E_2) = 1$ and $(r \circ k)_y$ valued in $\{0,1\}$.
  Now $E = E_1 \cap E_2$ is as required.
\end{proof}

\subsection{Main statement}\label{mainsec}

\begin{theorem}[synthetic ergodic decomposition theorem]\label{mainthm}
 Let $\cat{C}$ be a Markov category. 
 Let $X$ be a deterministic dynamical system in $\cat{C}$ with monoid $M$. Suppose that 
 \begin{itemize}
  \item The underlying object $X$ of $\cat{C}$ has disintegrations;
  \item The Markov colimit $X_\inv$ of the dynamical system exists.
 \end{itemize}

 Then every invariant state of $X$ can be written as a composition $k\circ q$ such that $k$ is $q$-almost surely ergodic. 
\end{theorem}

Let's now instantiate the theorem in $\cat{Stoch}$, recalling that Markov colimits of deterministic dynamical systems always exist (\Cref{invsigmacolimit}).

\begin{corollary}\label{mainstoch}
 Let $X$ be a deterministic dynamical system in $\Stoch$ with monoid $M$ (for example, with $M$ acting via measurable functions $m:X\to X$).
 Suppose that the measurable space $X$ satisfies a disintegration theorem (for example, if it is a standard Borel space). 
 
 Then every invariant measure on $X$ can be written as a mixture of ergodic states. 
\end{corollary}

Compare this with the traditional statements~\cite[Theorem~5.1.3]{viana-oliveira}, \cite[Proposition~4]{tao}.
Note also that the statement holds regardless of the cardinality or extra structure of the monoid $M$.

Now that all the necessary categorical setting is in place, the proof is very concise. 
Before looking at it, let's explain the intuition behind it a little.
Since the Markov colimit exists (for example, the invariant $\sigma$-algebra), we have a ``weak quotient'' map $r:X\to X_\inv$ which intuitively forgets the distinction between points that lie on the same orbit. 
We then construct a disintegration $r^+_p:X_\inv\to X$. Intuitively, this kernel maps each orbit back to a measure on $X$ which, is 
\begin{itemize}
 \item Supported on the given orbit (i.e.~$r^+$ is a stochastic section of $r$ almost surely);
 \item Uniform within the given orbit (i.e.~$r^+$ is almost surely ergodic).
\end{itemize}
This disintegration expresses then $p$ as a mixture of ergodic measures.

Let's now look at the proof. 

\begin{proof}[Proof of \Cref{mainthm}]
 Let $p:I\to X$ be an invariant state. Consider the map $r:X\to X_\inv$, and form the disintegration $r^+_p:X_\inv\to X$.
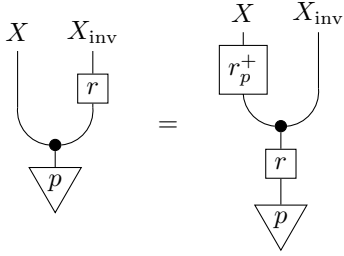
 $$
 \begin{tikzpicture}
	\begin{pgfonlayer}{nodelayer}
		\node [style=bn] (13) at (-3, -0.5) {};
		\node [style=none] (14) at (-4, 0.5) {};
		\node [style=morphism] (15) at (-2, 1) {$r$};
		\node [style=none] (16) at (-4, 2) {};
		\node [style=none] (17) at (-2, 2) {};
		\node [style=none] (18) at (-4, 2.5) {$X$};
		\node [style=none] (19) at (-2, 2.5) {$X_\inv$};
		\node [style=morphism] (21) at (3, -1) {$r$};
		\node [style=bn] (22) at (3, 0) {};
		\node [style=morphism] (23) at (2, 1.5) {$r^+_p$};
		\node [style=none] (24) at (4, 1) {};
		\node [style=none] (25) at (2, 2.5) {};
		\node [style=none] (26) at (4, 2.5) {};
		\node [style=none] (27) at (2, 3) {$X$};
		\node [style=none] (28) at (4, 3) {$X_\inv$};
		\node [style=none] (29) at (0, 0) {$=$};
		\node [style=state] (30) at (-3, -1.5) {$p$};
		\node [style=state] (31) at (3, -2.5) {$p$};
		\node [style=none] (32) at (-2, 0.5) {};
		\node [style=none] (33) at (2, 1) {};
	\end{pgfonlayer}
	\begin{pgfonlayer}{edgelayer}
		\draw (16.center) to (14.center);
		\draw (17.center) to (15);
		\draw [bend right=45] (14.center) to (13);
		\draw (25.center) to (23);
		\draw (26.center) to (24.center);
		\draw [bend right=45] (22) to (24.center);
		\draw (22) to (21);
		\draw (21) to (31);
		\draw (13) to (30);
		\draw [bend left=45] (32.center) to (13);
		\draw (15) to (32.center);
		\draw (33.center) to (23);
		\draw [bend right=45] (33.center) to (22);
	\end{pgfonlayer}
\end{tikzpicture}
 $$
 By marginalizing the equation above over $X_\inv$, we see that $p=r^+_p \circ r\circ p$, i.e.~we are decomposing $p$ into the composition of $r\circ p:I\to X_\inv$ followed by $r^+_p$. Now denote $r\circ p$ by $q$. 

 Let's show that $r^+_p$ is $q$-almost surely ergodic.
 To see that $r^+_p$ is $q$-almost surely left-invariant, note that for all $m\in M$, 
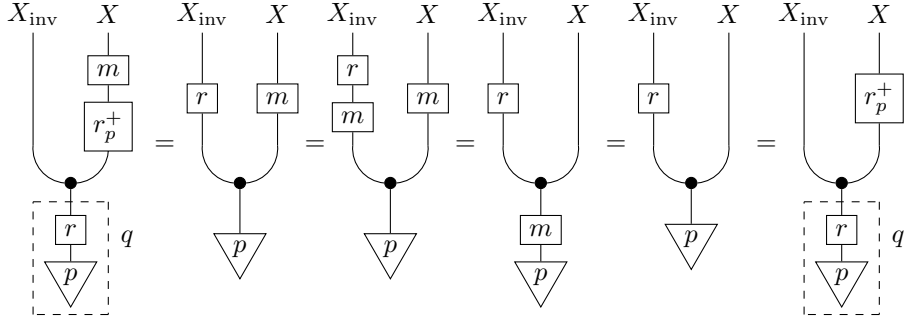
 $$
 \begin{tikzpicture}
	\begin{pgfonlayer}{nodelayer}
		\node [style=morphism] (21) at (-10.5, -2.25) {$r$};
		\node [style=bn] (22) at (-10.5, -1) {};
		\node [style=morphism] (23) at (-9.5, 0.5) {$r^+_p$};
		\node [style=none] (24) at (-11.5, 0) {};
		\node [style=none] (26) at (-11.5, 3) {};
		\node [style=none] (27) at (-9.5, 3.5) {$X$};
		\node [style=none] (28) at (-11.5, 3.5) {$X_\inv$};
		\node [style=none] (29) at (-8, 0) {$=$};
		\node [style=bn] (33) at (-6, -1) {};
		\node [style=morphism] (34) at (-9.5, 2) {$m$};
		\node [style=none] (35) at (-9.5, 3) {};
		\node [style=none] (38) at (-7, 3) {};
		\node [style=none] (39) at (-5, 3) {};
		\node [style=none] (40) at (-5, 3.5) {$X$};
		\node [style=none] (41) at (-7, 3.5) {$X_\inv$};
		\node [style=bn] (45) at (-2, -1) {};
		\node [style=morphism] (46) at (-3, 0.75) {$m$};
		\node [style=none] (48) at (-3, 3) {};
		\node [style=none] (49) at (-1, 3) {};
		\node [style=none] (50) at (-1, 3.5) {$X$};
		\node [style=none] (51) at (-3, 3.5) {$X_\inv$};
		\node [style=none] (52) at (0, 0) {$=$};
		\node [style=none] (53) at (-4, 0) {$=$};
		\node [style=morphism] (54) at (-3, 2) {$r$};
		\node [style=none] (58) at (-1, 0) {};
		\node [style=morphism] (59) at (-1, 1.25) {$m$};
		\node [style=morphism] (60) at (-7, 1.25) {$r$};
		\node [style=morphism] (61) at (-5, 1.25) {$m$};
		\node [style=none] (62) at (-7, 0) {};
		\node [style=none] (63) at (-5, 0) {};
		\node [style=morphism] (67) at (2, -2.25) {$m$};
		\node [style=bn] (68) at (2, -1) {};
		\node [style=none] (69) at (1, 0) {};
		\node [style=none] (70) at (3, 0) {};
		\node [style=morphism] (71) at (1, 1.25) {$r$};
		\node [style=none] (72) at (1, 3) {};
		\node [style=none] (73) at (3, 3) {};
		\node [style=none] (74) at (3, 3.5) {$X$};
		\node [style=none] (75) at (1, 3.5) {$X_\inv$};
		\node [style=none] (76) at (4, 0) {$=$};
		\node [style=bn] (81) at (6, -1) {};
		\node [style=none] (82) at (5, 3) {};
		\node [style=none] (83) at (7, 3) {};
		\node [style=none] (84) at (7, 3.5) {$X$};
		\node [style=none] (85) at (5, 3.5) {$X_\inv$};
		\node [style=morphism] (86) at (5, 1.25) {$r$};
		\node [style=none] (88) at (5, 0) {};
		\node [style=none] (89) at (7, 0) {};
		\node [style=none] (90) at (8, 0) {$=$};
		\node [style=morphism] (95) at (10, -2.25) {$r$};
		\node [style=bn] (96) at (10, -1) {};
		\node [style=morphism] (97) at (11, 1.25) {$r^+_p$};
		\node [style=none] (98) at (9, 0) {};
		\node [style=none] (99) at (9, 3) {};
		\node [style=none] (100) at (11, 3.5) {$X$};
		\node [style=none] (101) at (9, 3.5) {$X_\inv$};
		\node [style=none] (103) at (11, 3) {};
		\node [style=none] (104) at (11, 0) {};
		\node [style=none] (105) at (-9.5, 0) {};
		\node [style=none] (106) at (-3, 0) {};
		\node [style=state] (110) at (-10.5, -3.5) {$p$};
		\node [style=state] (111) at (-6, -2.75) {$p$};
		\node [style=state] (112) at (-2, -2.75) {$p$};
		\node [style=state] (113) at (2, -3.5) {$p$};
		\node [style=state] (114) at (6, -2.5) {$p$};
		\node [style=state] (115) at (10, -3.5) {$p$};
		\node [style=none] (116) at (-11.5, -1.5) {};
		\node [style=none] (117) at (-9.5, -1.5) {};
		\node [style=none] (118) at (-11.5, -4.5) {};
		\node [style=none] (119) at (-9.5, -4.5) {};
		\node [style=none] (120) at (-9, -2.5) {$q$};
		\node [style=none] (121) at (9, -1.5) {};
		\node [style=none] (122) at (11, -1.5) {};
		\node [style=none] (123) at (9, -4.5) {};
		\node [style=none] (124) at (11, -4.5) {};
		\node [style=none] (125) at (11.5, -2.5) {$q$};
	\end{pgfonlayer}
	\begin{pgfonlayer}{edgelayer}
		\draw (26.center) to (24.center);
		\draw [bend left=45] (22) to (24.center);
		\draw (22) to (21);
		\draw (34) to (23);
		\draw (35.center) to (34);
		\draw (48.center) to (54);
		\draw (54) to (46);
		\draw (38.center) to (60);
		\draw (39.center) to (61);
		\draw (49.center) to (59);
		\draw (59) to (58.center);
		\draw (61) to (63.center);
		\draw (60) to (62.center);
		\draw [bend right=45] (62.center) to (33);
		\draw [bend right=45] (33) to (63.center);
		\draw [bend right=45] (45) to (58.center);
		\draw (68) to (67);
		\draw (72.center) to (71);
		\draw (73.center) to (70.center);
		\draw (71) to (69.center);
		\draw [bend right=45] (69.center) to (68);
		\draw [bend right=45] (68) to (70.center);
		\draw (82.center) to (86);
		\draw (86) to (88.center);
		\draw [bend right=45] (88.center) to (81);
		\draw [bend right=45] (81) to (89.center);
		\draw (89.center) to (83.center);
		\draw (99.center) to (98.center);
		\draw [bend left=45] (96) to (98.center);
		\draw (96) to (95);
		\draw (103.center) to (97);
		\draw (104.center) to (97);
		\draw [bend left=45] (104.center) to (96);
		\draw (23) to (105.center);
		\draw [bend left=45] (105.center) to (22);
		\draw (106.center) to (46);
		\draw [bend right=45] (106.center) to (45);
		\draw (21) to (110);
		\draw (33) to (111);
		\draw (45) to (112);
		\draw (67) to (113);
		\draw (81) to (114);
		\draw (95) to (115);
		\draw [style=dashed box] (116.center) to (117.center);
		\draw [style=dashed box] (117.center) to (119.center);
		\draw [style=dashed box] (119.center) to (118.center);
		\draw [style=dashed box] (118.center) to (116.center);
		\draw [style=dashed box] (121.center) to (122.center);
		\draw [style=dashed box] (122.center) to (124.center);
		\draw [style=dashed box] (124.center) to (123.center);
		\draw [style=dashed box] (123.center) to (121.center);
	\end{pgfonlayer}
\end{tikzpicture}
 $$ 
 using, in order,
 the definition of $r^+_p$ as a disintegration, right-invariance of $r$, determinism of $m$, left-invariance of $p$, and again the definition of $r^+_p$ as a disintegration.
 
 By \Cref{ergasuni}, all that remains to be shown in order to prove $q$-almost sure ergodicity is that $r\circ r^+_p$ is $q$-almost surely deterministic.
 To see this, note that since $r$ is deterministic, we can apply \Cref{id-as} with $r$ in place of $f$. The proposition tells us that $r\circ r^+_p$ is $q$-almost surely equal to the identity, which is deterministic.
\end{proof}

As one can see, at this level all the measure-theoretic and analytic details are taken care of by the formalism, and one can focus on the conceptual reasoning.

\bibliographystyle{plain}
\bibliography{markov}

\appendix

\section{The invariant sigma-algebra as a weak quotient}\label{weakquotient}

Let's now explain why $X$ with the invariant $\sigma$-algebra can play the role of a space of orbits. (See \Cref{dynsyscat}, \Cref{markovquotient}, and \Cref{mainsec} for context.)

We first consider the intrinsic `indistinguishability relation' on a measurable space.
For a measurable space $(Y,\Sigma_Y)$, this is the equivalence relation given by $y\sim y'$ if and only if for every measurable set $A\in\Sigma_Y$, $y\in A$ if (and only if) $y'\in Y$.
The relation $\sim$ is discrete (coincides with equality) if, for example,
\begin{itemize}
\item $Y$ is a $T_0$ topological space (for example sober or Hausdorff) equipped with the Borel $\sigma$-algebra;
\item $Y$ has the property that all singletons are measurable.
\end{itemize}
In particular, $\sim$ is discrete if $Y$ is a standard Borel space.
 
\begin{proposition}
 The relation $\sim$ is the kernel of the map
 $$
 \begin{tikzcd}[row sep=0]
  Y \ar{r}{\delta} & PY \\
  y \ar[mapsto]{r} & \delta_y
 \end{tikzcd}
 $$
 assigning to each point $y$ the Dirac distribution over it. In particular, the relation $\sim$ is discrete if and only if the map $\delta$ is injective.
\end{proposition}

(See \cite{ours_LICS} for additional context on injectivity of the map $\delta$.)

\begin{proposition}
 The following conditions are equivalent for elements $y,y'$ of a measurable space $(Y,\Sigma_Y)$:
 \begin{itemize}
 \item $y\sim y'$;
 \item For every measurable real function $f:Y\to\R$, $f(y)=f(y')$;
 \item For every Markov kernel $k$ from $Y$ to a space $Z$ and every $C\in\Sigma_Z$, $k(C|y)=k(C|y')$.
\end{itemize}
\end{proposition}

Therefore, Markov kernels are ``blind'' to indistinguishable elements. In particular, two measurable spaces can be \emph{isomorphic in the category of deterministic Markov kernels} $\cat{Stoch}_\det$ \emph{even if they have different underlying sets}. The following proposition illustrates a canonical example.

\begin{proposition}\label{quotiso}
 Let $(Y,\Sigma_Y)$ be a measurable space, let $Y/_\sim$ be the quotient space w.r.t.~the indistinguishability relation, and denote by $q:Y\to Y/_\sim$ the quotient map $y\mapsto [y]$. 
 If we equip $Y/_\sim$ with the quotient $\sigma$-algebra, so that $q$ is measurable, the Markov kernel induced by $q$ is an isomorphism of $\cat{Stoch}_\det$.
 That is, the deterministic kernel induced by $q$ has a deterministic (i.e.~zero-one) inverse in the category of Markov kernels.
\end{proposition}

(In this section $q$ will denote a quotient map, rather than a state or measure.)

Note that an isomorphism in the category of deterministic kernels, even though the underlying sets may differ, does imply that the \emph{sigma algebras} are isomorphic. More on this shortly. 

Let's now return to dynamical systems. Let $X$ be a dynamical system in $\cat{Stoch}$ where the monoid $M$ acts by measurable functions $m : X \to X$ in $\cat{Meas}$ --- recall that this is a special case of a deterministic dynamical systems in our terminology. We will show that the traditional quotient and the space $X_\inv$ are isomorphic in $\cat{Stoch}$ via deterministic (in our terminology) kernels.

Denote by $\sim_M$ the equivalence relation on $X$ generated by the action of $M$. As usual, $x\sim_M y$ if and only if there exists a ``zig-zag'' connecting $x$ and $y$, explicitly, a finite sequence $m_1,\dots,m_n\in M$ and elements $x_0,x_1,\dots,x_n\in X$ with $x_0=x$, $x_n=y$, and such that
\begin{itemize}
 \item For odd $i$, $m_i(x_{i-1})=x_i$.
 \item For even $i$, $m_i(x_i)=x_{i-1}$.
\end{itemize}
For $n=6$, the situation is represented by the following picture. 
$$
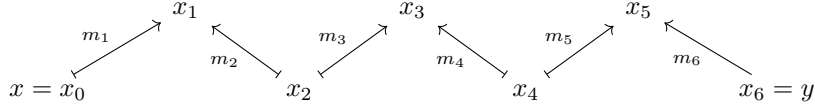
\begin{tikzcd}
 & x_1 && x_3 && x_5 \\
 x=x_0 \ar[mapsto]{ur}{m_1} && x_2 \ar[mapsto]{ul}{m_2} \ar[mapsto]{ur}{m_3} && x_4 \ar[mapsto]{ul}{m_4} \ar[mapsto]{ur}{m_5} && x_6=y \ar{ul}{m_6}
\end{tikzcd}
$$
(Note that such a zig-zag is necessary if $M$ is not a group.)

Now let $X/M$ be the quotient of $X$ w.r.t.~the relation $\sim_M$, and denote by $q:X\to X/M$ the quotient map. 
 If we equip $X/M$ with the quotient $\sigma$-algebra, so that $q$ is measurable, the outer triangle in the following diagram commutes for all $m\in M$. 
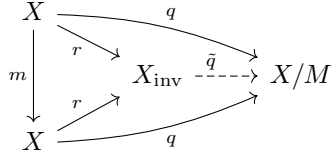
$$
   \begin{tikzcd}[row sep=small]
    X \ar{dd}[swap]{m} \ar{dr}[swap]{r} \ar[bend left=10]{drr}{q} \\
    & X_\inv \ar[dashed]{r}[near start]{\tilde q} & X/M \\
    X \ar{ur}{r} \ar[bend right=10]{urr}[swap]{q}
   \end{tikzcd}
$$
Therefore the map $q:X\to X/M$ is also measurable for the invariant $\sigma$-algebra, i.e.~it descends to $X_\inv$ as a measurable map $\tilde q$, making the diagram above commute.

\begin{theorem}\label{quotiso2}
 Let $X$ be a deterministic dynamical system in $\cat{Stoch}$ with monoid $M$ acting by measurable functions $m:X\to X$. 
 Construct the quotient map $q:X\to X/M$ as described above.
 Then the kernel $X_\inv\to X/M$ induced by the measurable map $\tilde q$ is an isomorphism of $\cat{Stoch}_\det$. That is, it has a deterministic (i.e.~zero-one) inverse in the category of Markov kernels. 
\end{theorem}

Note that this is the case regardless of the particular $\sigma$-algebra of the original space $X$, and regardless of the cardinality or structure of the monoid $M$.

\begin{lemma}\label{invzig}
 Assume the hypotheses of \Cref{quotiso2}. Let $A\subseteq X$ be an invariant set. Then if $x\sim_M y$, then $x\in A$ if and only if $y\in A$. 
\end{lemma}

\begin{proof}[Proof of \Cref{invzig}]
 Invariance of $A$ means that $x\in A$ if and only if $m(x)\in A$ for all $m\in M$. Now let $x\sim y$, so that there is a zig-zag given by $m_1,\dots,m_n\in M$ and $x=x_0,\dots,x_n=y$ as above. Then for all $i=1,\dots n$ (odd or even), we have $x_{i-1}\in A$ if and only if $x_i\in A$. Therefore $x=x_0\in A$ if and only if $y=x_n\in A$.
\end{proof}

\begin{lemma}\label{qbijection2}
 Under the hypotheses of \Cref{quotiso2}, the preimage map $q^{-1}:\Sigma_{X/M}\to \Sigma_{X_\inv}$ is a bijection.
\end{lemma}

\begin{proof}[Proof of \Cref{qbijection2}]
 First of all, $q^{-1}:\Sigma_{X/M}\to \Sigma_{X_\inv}$ is injective since $q:X_\inv\to X/M$ is surjective by construction.
 Moreover, we have that 
 \begin{equation}\label{qqA2}
 q^{-1}(q(A)) = A
 \end{equation}
 for each measurable invariant set $A\in\Sigma_{X_\inv}$. We have as usual $A\subseteq q^{-1}(q(A)$. For the reverse inclusion, let $x\in q^{-1}(q(A))$, i.e.~such that $[x]\in q(A)$.
 This means that there exists $a\in A$ with $[x]=[a]$, i.e.~$x\sim a$. But since $A$ is invariant, by \Cref{invzig}, if $a\in A$ and $x\sim_M a$, then $x\in A$ as well. 
\end{proof}

\begin{proof}[Proof of \Cref{quotiso2}]
 Construct the kernel $h:X/M\to X_\inv$ as follows, for each $[x]\in X/M$ and each invariant $A\in\Sigma_{X_\inv}$.
 $$
 h(A|[x]) \coloneqq 1_A(x) = \begin{cases}
                              1 & x\in A ; \\
                              0 & x\notin A .
                             \end{cases}
 $$
 This is well defined: by \Cref{invzig}, and since $A$ is invariant, for $x\sim_M y$, $x\in A$ if and only if $y\in A$. 
 
 To see measurability in $[x]$, it suffices to prove that the following subset of $X/M$ is measurable for each invariant set $A$.
 \begin{align*}
  h^*(A) &\coloneqq \{[x]\in X/M : h(A|[x]) = 1 \} \\
  &= \{[x]\in X/M : x\in A \} \\
  &= q(A) .
 \end{align*}
 By the definition of quotient $\sigma$-algebra, it suffices to show that $q^{-1}(q(A))$ is a measurable subset of $X$.  
 But now by \eqref{qqA2}, $q^{-1}(q(A)) = A$, which is measurable. So $q(A)$ is a measurable set, and hence $h$ is a measurable kernel. 
 
 To see that $h$ inverts $q$, notice that for every $x\in X$ and every measurable $A\in\Sigma_X$,
 $$
 h(A|q(x)) = h(A|[x]) = 1_A(x) ,
 $$
 which is the identity kernel. Just as well, for each $[x]\in X/M$ and $B\in\Sigma_{X/M}$,
 $$
 q_*h(B|[x]) = h(q^{-1}(B)|[x]) = 1_{q^{-1}(B)}(y) = 1_B([x]) ,
 $$
 once again the identity kernel. 
\end{proof}

Now, usually the orbit space plays the role of classifying invariant observables, in the sense that an observable is invariant if and only if it descends to the orbit space in a well defined way. In our case, the invariant $\sigma$-algebra takes care of this without the need of actually taking the quotient. 
In general, especially if the cardinality of $M$ is large, the set-theoretic quotient of $X$ can be very badly behaved as a measurable space. Considering the object $X_\inv$ in the Markov category $\cat{Stoch}$, instead, one avoids having to deal with ``bad'' quotients. 
It is still true that in general the $\sigma$-algebra of $X_\inv$ does not separate points, but this is less of a problem in $\cat{Stoch}$ and $\cat{Stoch}_\det$, since those categories do not really deal with points, but rather, with ``points up to indistinguishability''.
The situation is analogous to what happens in homotopy theory when one takes ``weak quotients'' or other similar constructions, such as resolutions.

\end{document}